\documentclass[11pt]{amsart}

\usepackage[T1]{fontenc}
\usepackage{lmodern}
\usepackage[letterpaper,margin=1in]{geometry}
\usepackage{graphicx}
\usepackage{amssymb,amsfonts}
\usepackage{mathtools}
\usepackage{bm}
\usepackage{mathrsfs}
\usepackage{enumitem}
\usepackage{microtype}
\microtypesetup{expansion=false}
\usepackage{booktabs}
\usepackage{array}
\usepackage{xcolor}
\usepackage{hyperref}
\usepackage{cleveref}

\hypersetup{
  colorlinks=true,
  linkcolor=blue!50!black,
  citecolor=blue!50!black,
  urlcolor=black,
  filecolor=black,
  pdftitle={Sharp Approximation Rates for Shallow ReLU-k Neural Networks on Critical Besov Classes},
  pdfauthor={Yupeng Wang}
}
\allowdisplaybreaks

\makeatletter
\def\@setkeywords{%
  {\bfseries Key words.}\enspace\@keywords}
\def\@setsubjclass{%
  {\bfseries MSC codes.}\enspace\@subjclass}
\def\@adminfootnotes{%
  \let\@makefnmark\relax \let\@thefnmark\relax
  \ifx\@empty\@date\else \@footnotetext{\@setdate}\fi
  \ifx\@empty\@keywords\else \@footnotetext{\@setkeywords}\fi
  \ifx\@empty\@subjclass\else \@footnotetext{\@setsubjclass}\fi
  \ifx\@empty\thankses\else \@footnotetext{%
    \def\par{\let\par\@par}\@setthanks}%
  \fi
}
\makeatother

\theoremstyle{plain}
\newtheorem{theorem}{Theorem}[section]
\newtheorem{proposition}[theorem]{Proposition}
\newtheorem{lemma}[theorem]{Lemma}
\newtheorem{corollary}[theorem]{Corollary}

\theoremstyle{definition}
\newtheorem{definition}[theorem]{Definition}

\theoremstyle{remark}
\newtheorem{remark}[theorem]{Remark}

\numberwithin{equation}{section}
\numberwithin{table}{section}
\numberwithin{figure}{section}

\newcommand{\R}{\mathbb R}
\newcommand{\N}{\mathbb N}
\newcommand{\Sph}{\mathbb S}

\newcommand{\supp}{\operatorname{supp}}

\newcommand{\eps}{\varepsilon}

\title[Sharp Besov Approximation Rates]
{Sharp Approximation Rates for Shallow $\text{ReLU}$ Neural Networks on Critical Besov Classes}
\author{Yupeng Wang}
\date{}
\keywords{Critical Besov spaces, shallow \texorpdfstring{$\operatorname{ReLU}^k$}{ReLU-k} neural networks,
variation spaces, nonlinear $n$-term approximation, sharp approximation rates}
\subjclass[2020]{46E35, 41A46, 41A25}

\begin{document}
\begin{abstract}
Let $\mathbb D$ be the normalized ridge dictionary generated by
$\operatorname{ReLU}^k$ on a bounded Lipschitz domain
$\Omega\subset\mathbb R^d$. We determine the sharp algebraic rate of
finite $n$-term approximation from $\mathbb D$ when the outer
$\ell^1$ coefficient budget is independent of $n$. More precisely, let
$d\ge3$, $k\in\mathbb N_+$, $0\le m\le k$, $0<p<1$, and $0<q\le1$.
For the unit ball of the critical Besov space
$B_{p,q}^{k+d/p}(\Omega)$, measured in $H^m(\Omega)$, the optimal
algebraic exponent is
\[
 \min\left\{\frac{k-m+d/2}{d-1},
 \frac{k-m+d/2+1/p-1/2}{d}\right\}.
\]
We prove two-sided estimates in the three regimes determined by the two
branches of this minimum and give the corresponding logarithmic factors at
and beyond the transition. The estimates coincide without logarithmic loss
in the strict angular regime and at the transition when
$q\le[1/2+(k-m+d/2)/(d-1)]^{-1}$. The upper estimate combines critical wavelet sparsity,
localized Fourier--Radon representations, and a stable allocation of
directions and biases. The lower estimates arise from two distinct
obstructions, namely radial ridge approximation on the direction sphere and
Gevrey localization in the joint direction--bias space. The critical
smoothness is exactly the endpoint at which Besov regularity provides a
uniform shallow-network variation bound without additional scale decay.
The resulting representation exponent is strictly larger than the
degree-limited exponent for fixed-degree isotropic finite elements under a
parameter-count comparison. This last statement concerns best approximation,
not training or computational complexity.
\end{abstract}

\maketitle

\section{Introduction and main result}\label{sec:introduction}

Neural-network-based PDE solvers form an important class of nonlinear
numerical methods for partial differential equations. They replace a
prescribed mesh or linear trial space by a trainable network class and are
therefore attractive when conventional discretizations become expensive.
Their best achievable accuracy, however, is limited by how efficiently the
underlying solution can be represented with controlled width and coefficient
size. A functional-analytic description of that representation power is thus
needed before optimization, quadrature, or solver errors can be assessed
\cite{EYu2018,Xu2020}.

This paper studies that question for finite shallow
$\operatorname{ReLU}^k$ networks. The directions are normalized, the biases
range over a fixed compact interval, and the sum of the absolute outer
coefficients is bounded independently of the width. For a target in
$\operatorname{ReLU}^k$ networks. Let $\mathbb D$ denote the ridge dictionary
with normalized directions and biases in a fixed compact interval, and let
$\mathcal L_1(\mathbb D)$ denote its measure-valued variation space. The sum
of the absolute outer coefficients is bounded independently of the width.
For a target in
$B_{p,q}^{k+d/p}(\Omega)$, with $0<p<1$ and $0<q\le1$, we ask how fast the
best $H^m(\Omega)$ error can decay as the number of active ridge atoms grows.
The fixed coefficient budget is essential: it connects the finite networks
to the measure-valued variation space while preventing the rate from being obtained through coefficients whose mass diverges with the width.

Besov spaces provide the natural source scale for this problem. Their
quasi-norms capture smoothness and the spatial sparsity of corner, edge, and
multiscale singularities. They also govern best $n$-term wavelet
approximation and the approximation classes underlying adaptive finite element methods \cite{DeVore1998,GaspozMorin2014,Gantumur2017}. At the
critical index, the companion paper \cite{LiWangSIMA2026} proves
\[
 B_{p,q}^{k+d/p}(\Omega)\hookrightarrow\mathcal L_1(\mathbb D),
 \qquad 0<p<1,\quad 0<q\le1,
\]
for the normalized ridge dictionary used here. This embedding supplies an
infinite-width representation with uniformly bounded variation norm. It does
not determine the sharp finite-width error, the transition between the
direction and direction--bias regimes, or a matching lower bound under the
same coefficient budget. Resolving these three points is the purpose of the
present paper.

Two further developments clarify the surrounding approximation-theoretic
picture.  Approximation classes for adaptive methods can be characterized by
Besov smoothness under suitable refinement models
\cite{BinevDahmenDeVorePetrushev2002}, while domain-intrinsic weighted
variation spaces retain dimension-robust shallow-ReLU approximation rates on
bounded domains \cite{DeVoreNowakParhiSiegel2025}.  Entropy estimates also
lead to convergence guarantees for greedy nonlinear dictionary approximation
\cite{LiSiegel2024}.  These results motivate keeping separate the three issues
addressed here: membership in a representation space, algorithmic selection,
and the sharp best finite-width rate under a fixed coefficient budget.

The result also permits a precise comparison with fixed-degree finite
elements. Consider shape-regular isotropic simplicial meshes and polynomial
degree $k$. If $N$ denotes the number of elements, then $N$ is proportional
to the number of degrees of freedom when $d$ and $k$ are fixed. Classical
local polynomial approximation imposes the degree-limited benchmark
\[
 N^{-(k+1-m)/d}
\]
in $H^m(\Omega)$, or in the corresponding broken norm when the finite element
space is not $H^m$-conforming \cite{Ciarlet2002,LinXieXu2014}. Besov and
multilevel characterizations identify the targets that attain a prescribed
best-mesh rate, but they do not remove the local polynomial degree barrier.
The critical class considered here embeds continuously into
$H^{k+1}(\Omega)$ when $d\ge3$, since
$B_{p,q}^{k+d/p}\hookrightarrow B_{2,q}^{k+d/2}$ and $k+d/2>k+1$.
Consequently, the classical fixed-degree benchmark applies and becomes
$N^{-k/d}$ in the $H^1$ energy norm. Rate-optimal AFEM theory compares an
adaptive algorithm with its best-mesh approximation class, usually for the
energy error together with data oscillation
 \cite{BinevDahmenDeVore2004,Stevenson2005,Stevenson2007,CasconEtAl2008}.
Related quasi-optimality results for adaptive mixed methods show that the
appropriate benchmark also depends on the variational norm and the variables
being controlled \cite{Li2021Natural}.

Our answer gives a direct comparison with this fixed-degree finite element
benchmark.  Up to the logarithmic factors in the second and third cases, the
network exponent is
\[
 \Gamma_{p,m}=\min\left\{
 \frac{k-m+d/2}{d-1},
 \frac{k-m+d/2+1/p-1/2}{d}
 \right\}.
\]
For $d\ge3$, $0<p<1$, and $0\le m\le k$, both candidate exponents exceed
the fixed-degree finite element value.  Indeed,
\[
 \frac{k-m+d/2}{d-1}-\frac{k+1-m}{d}
 =\frac{d(d-2)/2+k-m+1}{d(d-1)}>0,
\]
and
\[
 \frac{k-m+d/2+1/p-1/2}{d}-\frac{k+1-m}{d}
 =\frac{(d-3)/2+1/p}{d}>0.
\]
At $p=p_\ast$ the network rate has algebraic part $n^{-A_m}$, whereas
for $p_\ast<p<1$ it has algebraic part $n^{-G_{p,m}}$.  The logarithmic
gaps in these two cases do not alter the comparison because every fixed
power of $\log n$ is dominated by $n^\varepsilon$.  Consequently, after
identifying $n$ neurons and $N$ finite element degrees of freedom only at the
level of fixed-dimensional parameter counts, the best shallow-network error
has a strictly larger algebraic exponent on the critical Besov ball.  This
gain reflects the global hyperplane geometry and nonlinear parameter
selection of the ridge dictionary.  It is not a claim that a trained network
solves a PDE faster than an AFEM code.  The present analysis excludes
optimization error, numerical integration, conditioning, linear or nonlinear
solver work, and the cost of error certification.  It also does not compare
against anisotropic or $hp$-adaptive finite elements.
A recent bit-complexity analysis further cautions that a favorable exponent
in width or parameter count need not translate into a smaller finite-precision
information cost \cite{MaoXuBit2026}.  Thus the comparison above concerns
best representation error under the stated parameterization and nothing
stronger.

Two structural features explain why the critical index is delicate.
Rychkov's extension theorem transfers the Littlewood--Paley and
coefficient-space description of Besov regularity to the Lipschitz domains
used here \cite{Rychkov1999}. At the same time, an $\ell^p$-normalized family
of bumps at scale $h$ has uniformly bounded $B_{p,q}^{k+d/p}$ quasi-norm,
while its $H^m$ norm still records that spatial scale. Critical coefficient
sparsity and the geometry of the ridge dictionary therefore enter the
finite-width problem at the same order.

Barron's classical estimate \cite{Barron1993} gives dimension-robust Monte
Carlo approximation in $L^2$ for suitable Fourier moment classes, with the
output-layer $\ell^1$ mass controlled by an integral representation norm.
Random approximation and entropy arguments were developed further by
Makovoz \cite{Makovoz1996}.  Convex infinite-width formulations and their
output-weight regularization appear in \cite{Bach2017}, while the modern
function-space interpretation of Barron-type models is developed in
\cite{EMaWu2022}.  For ReLU and squared-ReLU ridge dictionaries,
Klusowski and Barron obtained estimates with simultaneous $\ell^1$ and
sparsity control \cite{KlusowskiBarron2018}.  These works motivate keeping
the coefficient budget explicit rather than treating it as an unspecified
width-dependent constant.

Approximation by ridge profiles nevertheless has geometric limitations.
The classical theory and its neural-network interpretation are surveyed in
\cite{Pinkus1999}, and Petrushev established direct approximation estimates
for ridge functions and neural networks \cite{Petrushev1998}.  Lower bounds
for Sobolev-type classes and families of ridge targets reveal the role of the
$(d-1)$-dimensional direction space
\cite{MaiorovMeirRatsaby1999,KLM2008}.  These angular obstructions also apply
to a fixed $\operatorname{ReLU}^k$ activation.

Recent work has sharpened this quantitative theory.  Siegel and Xu established
dimension-independent rates for broad classes of activations and showed how
stratified sampling can improve the basic sampling estimate
\cite{SiegelXu2020}.  They also obtained high-order $H^m$ approximation rates
for spectral Barron classes and
$\operatorname{ReLU}^k$ dictionaries \cite{SiegelXu2022}.  Their high-order
construction does not impose a uniform $\ell^1$ bound on the output
coefficients, and whether the same rate holds with such a bound is left open
there.  They subsequently placed neural representation spaces in a common
dictionary-variation framework \cite{SiegelXu2023} and proved sharp
approximation, entropy, and width estimates for smooth parameterized
dictionaries \cite{SiegelXu2024}.  Uniform-norm and derivative estimates for
variation balls are studied in \cite{MaSiegelXu2022,Siegel2025}.  Yang and
Zhou derived optimal shallow $\operatorname{ReLU}^k$ rates on classical
smoothness classes \cite{YangZhou2025}, and Mao, Siegel, and Xu used
Radon-transform methods to establish nearly optimal Sobolev approximation
rates over a broad range of indices \cite{MaoSiegelXu2026}.  A general
account of approximation, stability, and rate-distortion questions for
neural-network manifolds is given in \cite{DeVoreHaninPetrova2021}.

A complementary line of work fixes the inner parameters and optimizes only
the outer coefficients.  Liu, Mao, and Xu prove the rate
$n^{-(r-m)/d}$ in $H^m$ for deterministic well-distributed directions and
biases, with a logarithmic loss for independent uniform sampling
\cite{LiuMaoXu2025}.  If their theorem is applied to the present source class
only through the Sobolev embedding
$B_{p,q}^{k+d/p}(\Omega)\hookrightarrow H^{k+d/2}(\Omega)$, it yields
$n^{-\tau_m/d}$, where $\tau_m=k-m+d/2$.  In this specialization, the
coefficient budget in that construction grows like $n^{1/(2d)}$, whereas
the present theorem gives the strictly larger algebraic exponent
$\Gamma_{p,m}>\tau_m/d$ with a budget independent of $n$.  The comparison
isolates the extra information carried by the $p<1$ Besov sparsity: it allows
spatially concentrated targets that need not possess the additional half
Sobolev derivative required to obtain the same exponent from the linearized
Sobolev result.  It is not a same-class dominance statement.  On the sphere,
Mao and Xu prove that quasi-uniform linearized $\operatorname{ReLU}^k$
schemes saturate above the critical Sobolev index
\cite{MaoXuSaturation2025}.  That saturation result concerns fixed centers
and therefore does not furnish a lower bound for the adaptive nonlinear
dictionary studied here.

The present paper treats the critical quasi-Banach class
$B_{p,q}^{k+d/p}(\Omega)$ with $0<p<1$ and $0<q\le1$.  The approximants use
normalized directions, biases in a fixed compact interval, and an
output-layer $\ell^1$ budget independent of the width.  This formulation
differs from approximation of a variation ball and from models that permit
an unrestricted polynomial correction.  It exposes two independent lower
bound mechanisms.  The first is an angular obstruction governed by the
direction sphere.  The second is a direction--bias packing obstruction that
couples the finite parameter budget to critical Besov sparsity.

Building on the embedding and localized Fourier--Radon estimates of
\cite{LiWangSIMA2026}, we use those functional-analytic results as black-box
inputs. The remaining task is to convert a width-independent variation
budget into a sharp finite-width rate and to establish matching algebraic
lower bounds under the same budget. The two
branches of $\Gamma_{p,m}$ meet at the threshold introduced in
\cref{sec:finite-dictionary}.  The main theorem states the matching upper and
lower bounds under the same budget. The two algebraic mechanisms lead to
different rate branches, which meet at the threshold introduced in
\cref{sec:finite-dictionary}. The main theorem states the matching upper and
lower rates separately in the three parameter regimes, with explicit
logarithmic factors at the transition and in the direction--bias regime.

The paper makes three contributions.  First, it formulates the problem for
the finite class $\Sigma_{n,M}$ and proves a stable upper estimate with a
coefficient budget independent of $n$.  Second, it combines a radial angular
obstruction with a localized Gevrey-packet and cell-counting argument to identify
the algebraic phase transition along the critical line.  Finally, it realizes every polynomial of degree
at most $k$ by a fixed number of dictionary atoms, which shows that the
augmented and non-augmented formulations have the same algebraic upper rate.
We first explain this scope, then state the result precisely. The
technical ingredients are developed after the main statement.

\paragraph{Why the critical line is the focus.}
The relation $s=k+d/p$ is a structural endpoint for the present
fixed-budget problem, not a normalization imposed on an otherwise
unchanged theorem. The forward embedding of
unchanged theorem. Rychkov's extension theorem makes the
Littlewood--Paley description available on the Lipschitz domains considered
here \cite{Rychkov1999}. The forward embedding of
\cite[Theorem~1.1]{LiWangSIMA2026} gives uniform variation control at
$s=k+d/p$ for $0<q\le1$, and its rescaled-bump construction proves
that the smoothness threshold cannot be lowered
\cite[Section~4.1]{LiWangSIMA2026}. At scale $h_j=2^{-j}$, a
unit-amplitude localized atom has variation cost $O(h_j^{-k})$,
whereas a $B_{p,q}^s$ wavelet coefficient vector carries the weight
$h_j^{d/p-s}$. Thus the total variation cost is bounded by
\[
 \sum_j h_j^{s-k-d/p}
 \left(h_j^{d/p-s}\|(\lambda_{j,\nu})_\nu\|_{\ell^p}\right).
\]
Exactly on the critical line the geometric factor is one. The fine index
$q\le1$, spatial coefficient sparsity, and the distribution of the width
across scales must then be used simultaneously. This endpoint balance
is responsible for the transition in the main theorem and for its
explicit logarithmic terms.

The neighboring smoothness regimes pose different questions. If
$s<k+d/p$, the entire Besov unit ball has no uniform variation bound of
the above form. This does not say that individual functions lack ridge
approximants, or that approximation with a growing coefficient budget is
impossible. It says that uniform representability under a fixed budget
must first be addressed separately, even when the source class is
embedded in the error space. If $s>k+d/p$, the additional geometric
decay yields stronger summability. The embedding
$B_{p,q}^{s}(\Omega)\hookrightarrow B_{p,q}^{k+d/p}(\Omega)$ immediately
transfers our upper estimates to that smaller class, but it does not
transfer the matching lower estimates. Determining the optimal rate
there requires a new balance between excess smoothness, scale
allocation, and the fixed activation degree. We therefore establish a
sharp algebraic theory on the critical line and make no optimality
claim away from it.

This choice also specifies the PDE scope. Besov regularity can encode
spatial sparsity more accurately than a single Hilbert--Sobolev index,
but the condition $s-d/p=k$ remains a substantive assumption. It is not
automatically satisfied by every corner or interface singularity, and
must be checked for the particular solution before applying the theorem.

\subsection{Notation and statement of the main result}\label{sec:finite-dictionary}

Throughout, $\Omega\subset\R^d$ is a bounded Lipschitz domain,
$d\ge2$, $k\in\N_+$, and $m\in\N_0$ with $m\le k$. We set
$\sigma_k(t)=(t)_+^k$, where $(t)_+=\max\{t,0\}$.
The dimension restriction $d\ge3$ is imposed in the upper and main
theorems. Every nonempty open domain contains the interior balls and
closed cylinders used in the lower constructions, so these do not
constitute additional geometric hypotheses.

We use the real Sobolev norm
\[
 \|f\|_{H^m(\Omega)}^2
 =\sum_{|\alpha|\le m}\|D^\alpha f\|_{L^2(\Omega)}^2.
\]
For a fixed smooth inhomogeneous Littlewood--Paley resolution
$\{\Delta_j\}_{j\ge0}$, the Besov quasi-norm is
\[
 \|F\|_{B_{p,q}^s(\R^d)}
 =\left\|\bigl(2^{js}\|\Delta_jF\|_{L^p(\R^d)}\bigr)_{j\ge0}
   \right\|_{\ell^q}.
\]
On $\Omega$ we take the infimum of this quasi-norm over extensions
$F|_\Omega=f$. Standard equivalent resolutions give the same space
\cite{Triebel1983}. The Fourier convention is
$\widehat F(\xi)=\int_{\R^d}e^{-ix\cdot\xi}F(x)\,dx$.
The symbol $\Pi_k$ denotes polynomials of total degree at most $k$.
We write $a\lesssim b$ for an inequality up to a constant independent
of widths, scales, and the target function, and $a\simeq b$ when both
inequalities hold. Constants may depend on $d,k,m,p,q,\Omega$, fixed
cutoffs, and prescribed finite differentiation orders. In Gevrey lower
bounds they may also depend on the chosen $\vartheta$ and coefficient
budget. All limits and integrals in function spaces use their norm
topology unless stated otherwise.

Choose once and for all
\[
 c<\inf_{x\in\Omega,\,|\omega|=1}\omega\cdot x-1,
 \qquad
 d_0>\sup_{x\in\Omega,\,|\omega|=1}\omega\cdot x+1.
\]

The normalized shallow $\operatorname{ReLU}^k$ dictionary is
\[
 \mathbb D:=\left\{\sigma_k(\omega\cdot x-b):
 \omega\in\mathbb S^{d-1},\ b\in[-c,c]\right\}.
\]
Thus the sign and range of the bias and the normalization $|\omega|=1$
are fixed throughout.

\begin{definition}[Finite dictionary class and best $n$-term error]
For $n\in\N_0$ and $M\in[0,\infty]$ define
\[
\Sigma_{n,M}
:=\left\{
\sum_{r=1}^{N}a_r\sigma_k(\omega_r\cdot x-b_r):
 N\in\{0,\ldots,n\},\ |\omega_r|=1,\ b_r\in[-c,c],\ \sum_{r=1}^{N}|a_r|\le M
\right\}.
\]
The empty sum is zero, so $\Sigma_{n,0}=\{0\}$.
When $M=\infty$ the coefficient constraint is omitted.
The dictionary is continuously parameterized. ``Finite'' refers to the
number of active atoms, not to a fixed finite list of allowed parameters.  For $f\in H^m(\Omega)$ put
\[
\sigma_{n,m}(f;M)
:=\inf_{g\in\Sigma_{n,M}}\|f-g\|_{H^m(\Omega)},
\]
and for a quasi-Banach space $X\hookrightarrow H^m(\Omega)$ define
\[
E_{n,m}(B_X;M)
:=\sup_{\|f\|_X\le 1}\sigma_{n,m}(f;M).
\]
\end{definition}

The corresponding variation space consists of the functions
$f\in H^m(\Omega)$ admitting a representation
\[
 f=\int_{\mathbb S^{d-1}\times[-c,c]}
 \sigma_k(\omega\cdot x-b)\,\mathrm d\mu(\omega,b)
 \quad\text{in }H^m(\Omega).
\]
Its norm is
\[
 \|f\|_{\mathcal L_1(\mathbb D)}
 =\inf\{\|\mu\|_{\mathrm{TV}}:\mu\text{ represents }f\}.
\]

Throughout the critical Besov analysis we put
\[
s_0:=k+\frac dp,
\qquad
\tau_m:=k-m+\frac d2,
\qquad
\beta_m:=k-m+\frac12,
\]
\[
A_m:=\frac{\tau_m}{d-1},
\qquad
G_{p,m}:=\frac{\tau_m+1/p-1/2}{d},
\qquad
\Gamma_{p,m}:=\min\{A_m,G_{p,m}\}.
\]
The threshold where the two branches meet is
\[
p_\ast
:=\frac{d-1}{d-1+\beta_m}.
\]
Indeed $A_m=G_{p_\ast,m}$.

Set $L_n=1+\log(2+n)$ and $a_+=\max\{a,0\}$.  The main result can now be stated without any
additional model notation.

\begin{theorem}[Sharp approximation rates in three regimes]\label{thm:main-sharp}
Assume $d\ge3$, $0<p<1$, and $0<q\le1$, with $k,m,\Omega$ as above.
There is a fixed $0<M_0<\infty$, independent of $n$ and $\vartheta$,
such that the following three cases hold for all integers $n\ge2$.  In the bounds below, $E_{n,m}$
abbreviates $E_{n,m}(B_{B_{p,q}^{k+d/p}(\Omega)};M_0)$.  The constant
$c>0$ and all implicit constants are independent of $n$:
\begin{enumerate}[label=(\roman*)]
\item if $0<p<p_\ast$, then $E_{n,m}\simeq n^{-A_m}$;
\item if $p=p_\ast$, then
\[
 cn^{-A_m}\le E_{n,m}
 \lesssim n^{-A_m}L_n^{(A_m+1/2-1/q)_+};
\]
\item if $p_\ast<p<1$, then, for every fixed
$0<\vartheta<1$,
\[
 c_\vartheta n^{-G_{p,m}}L_n^{-\tau_m/\vartheta}
 \le E_{n,m}
 \lesssim n^{-G_{p,m}}
 L_n^{\frac{1/p-1}{1/p_\ast-1}
 (A_m+1/2-1/q)_+}.
\]
\end{enumerate}
\end{theorem}

The theorem separates two geometric regimes.  When
$p\le p_\ast$, the direction sphere is the limiting parameter set and
the angular exponent $A_m$ is sharp.  When $p>p_\ast$, the hard family
uses both directions and biases, and the exponent changes to $G_{p,m}$.  The
logarithmic factors measure the cost of balancing the dyadic scales in the
upper construction and excluding the packet tails in the lower construction;
they do not change the algebraic exponent. At $p=p_\ast$ the two
bounds match exactly if $q\le(A_m+1/2)^{-1}$. Outside this endpoint
subrange and the strict angular regime, the theorem does not assert
optimality of the logarithmic factors.

\paragraph{Comparison with adaptive finite elements.}
Theorem~\ref{thm:main-sharp} permits a direct comparison with fixed-degree
finite elements. Consider shape-regular isotropic simplicial meshes and
polynomial degree $k$. If $N$ denotes the number of elements, then $N$ is
proportional to the number of degrees of freedom when $d$ and $k$ are fixed.
Classical local polynomial approximation imposes the degree-limited
benchmark
\[
 N^{-(k+1-m)/d}
\]
in $H^m(\Omega)$, or in the corresponding broken norm when the finite element
space is not $H^m$-conforming \cite{Ciarlet2002,LinXieXu2014}. Besov and
 multilevel characterizations identify the targets that attain a prescribed
best-mesh rate \cite{BinevDahmenDeVorePetrushev2002}, but they do not remove
the local polynomial degree barrier.
The critical class in the theorem embeds continuously into
$H^{k+1}(\Omega)$ because
$B_{p,q}^{k+d/p}\hookrightarrow B_{2,q}^{k+d/2}$, $q\le1\le2$, and
$k+d/2>k+1$ for $d\ge3$. Consequently, the fixed-degree benchmark applies
and becomes $N^{-k/d}$ in the $H^1$ energy norm. Rate-optimal AFEM theory
compares an adaptive algorithm with its best-mesh approximation class,
usually for the energy error together with data oscillation
 \cite{BinevDahmenDeVore2004,Stevenson2005,Stevenson2007,CasconEtAl2008}.
For mixed formulations, quasi-optimal adaptive methods controlling the natural
variational norm provide a complementary benchmark
\cite{Li2021Natural}.  These algorithmic results compare an AFEM output with
its best admissible mesh class; they do not compare that mesh class with a
ridge dictionary.

Up to the logarithmic factors in cases (ii) and (iii), the exponent in
Theorem~\ref{thm:main-sharp} is $\Gamma_{p,m}$. Both branches defining this
exponent exceed the fixed-degree finite element value. Indeed,
\[
 A_m-\frac{k+1-m}{d}
 =\frac{d(d-2)/2+k-m+1}{d(d-1)}>0,
\]
and
\[
 G_{p,m}-\frac{k+1-m}{d}
 =\frac{(d-3)/2+1/p}{d}>0.
\]
The logarithmic factors do not change this algebraic comparison because
every fixed power of $\log n$ is dominated by $n^\varepsilon$. Hence, after
identifying $n$ neurons and $N$ finite element degrees of freedom only at the
level of fixed-dimensional parameter counts, the best shallow-network error
has a strictly larger algebraic exponent on the critical Besov ball. This is
a comparison of representation classes, not a claim that a trained network
solves a PDE faster than an AFEM code. It excludes optimization, quadrature,
conditioning, solver work, and error certification, and it does not cover
anisotropic or $hp$-adaptive finite elements. A recent bit-complexity
analysis further cautions that a favorable width exponent need not imply a
smaller finite-precision information cost \cite{MaoXuBit2026}.

\paragraph{Comparison with related shallow-network results.}
Siegel and Xu first obtained dimension-independent shallow-network rates for
general activations and improved the sampling exponent under additional
regularity by stratification \cite{SiegelXu2020}.  They later obtained
high-order $H^m$ approximation rates for spectral
Barron classes and $\operatorname{ReLU}^k$ dictionaries
\cite{SiegelXu2022}. That construction leaves open whether its high-order
rate persists under a uniform output-layer $\ell^1$ bound.
Theorem~\ref{thm:main-sharp} imposes such a budget, although on a different
Besov source class. A common dictionary-variation framework for neural
representation spaces is developed in \cite{SiegelXu2023}, while sharp
approximation, entropy, and width estimates for smooth dictionaries appear in
\cite{SiegelXu2024}.  The weighted variation spaces of DeVore, Nowak, Parhi,
and Siegel make the representation model intrinsic to a bounded domain while
preserving shallow-ReLU approximation rates
\cite{DeVoreNowakParhiSiegel2025}.  This domain-dependent construction is
distinct from the compact normalized $\operatorname{ReLU}^k$ dictionary used
here. Uniform-norm and derivative estimates for variation
balls are studied in \cite{MaSiegelXu2022,Siegel2025}. Optimal shallow
$\operatorname{ReLU}^k$ rates on classical smoothness classes appear in
\cite{YangZhou2025}, while Radon-transform methods yield nearly optimal
Sobolev approximation rates over a broad range of indices in
\cite{MaoSiegelXu2026}. A general account of approximation, stability, and
rate-distortion questions for neural-network manifolds is given in
\cite{DeVoreHaninPetrova2021}.

Entropy-based estimates for orthogonal greedy approximation provide an
algorithmic route from the metric entropy of a symmetric convex hull to a
convergence rate \cite{LiSiegel2024}.  The present theorem instead determines
the best $n$-term error over a Besov source ball.  It therefore supplies a
benchmark for, but not a convergence theorem about, greedy training.

A complementary approach fixes the inner parameters and optimizes only the
outer coefficients. Liu, Mao, and Xu prove the rate $n^{-(r-m)/d}$ in $H^m$
for deterministic well-distributed directions and biases, with a logarithmic
loss for independent uniform sampling \cite{LiuMaoXu2025}. Applied to the
present source class only through
$B_{p,q}^{k+d/p}(\Omega)\hookrightarrow H^{k+d/2}(\Omega)$, their theorem
gives $n^{-\tau_m/d}$. In this specialization, the coefficient budget in
their construction grows like $n^{1/(2d)}$, whereas
Theorem~\ref{thm:main-sharp} gives
$\Gamma_{p,m}>\tau_m/d$ with a budget independent of $n$. This comparison
isolates the information carried by $p<1$ Besov sparsity, which allows
spatially concentrated targets without the additional half Sobolev derivative
needed to obtain the same exponent from the linearized Sobolev result. The
different source classes preclude a universal dominance claim. On the sphere,
Mao and Xu prove saturation above the critical Sobolev index for quasi-uniform
linearized $\operatorname{ReLU}^k$ schemes \cite{MaoXuSaturation2025}. This
fixed-center conclusion is not a lower bound for the adaptive dictionary in
Theorem~\ref{thm:main-sharp}.

\paragraph{Contribution and proof structure.}
The finite-width theorem has three technical components. First, the upper
construction preserves a coefficient budget independent of $n$. Second, a
radial angular obstruction and a localized Gevrey-packet cell-counting
argument identify the algebraic phase transition. Third, every polynomial
of degree at most $k$ is realized by a fixed number of dictionary atoms, so
the augmented and non-augmented formulations have the same algebraic upper
rate.

The proof has an upper and a lower component.  For the upper bound, a
critical wavelet expansion reduces the problem to approximating localized
unit-amplitude atoms, after which a width allocation argument produces one
network with a coefficient budget independent of $n$.  For the lower bound,
a radial comparison theorem gives the angular obstruction, while a localized
Gevrey packet transform converts direction--bias separation into an
$H^m$ error bound.

The rest of the paper follows the upper--lower logic of the theorem.
Section~\ref{sec:upper-lemmas} develops the common analytic ingredients for
the upper estimate, and Section~\ref{sec:upper-proof} assembles them into a
stable network in the three regimes $p<p_\ast$,
$p=p_\ast$, and $p>p_\ast$.  Section~\ref{sec:lower-lemmas}
collects the angular and Gevrey localization tools for the lower estimate.
Section~\ref{sec:lower-proof} then proves the two lower branches and completes
the proof of Theorem~\ref{thm:main-sharp}.

\section{Lemmas for the upper bound}\label{sec:upper-lemmas}

This section isolates the ingredients that are common to all three upper
regimes.  We first absorb the finite-dimensional polynomial correction and
record the critical scaling identities.  The Besov--variation embedding and
its localized Fourier--Radon proof are invoked from
\cite{LiWangSIMA2026}; we retain only the compact wavelet formulation and
the discretization steps needed for a coefficient budget independent of the
width.  The case-dependent allocation of neurons is postponed to
Section~\ref{sec:upper-proof}.

\paragraph{Compatibility of variation norms.}

\begin{lemma}[Compatibility of the $L^2$ and $H^m$ variation spaces]
\label{lem:variation-topology}
The parameter map
\[
 T:\Theta:=\mathbb S^{d-1}\times[-c,c]\longrightarrow H^m(\Omega),
 \qquad T(\omega,b)=\sigma_k(\omega\cdot x-b),
\]
is continuous and uniformly bounded.  Every finite signed Borel measure on
$\Theta$ therefore defines an $H^m$ Bochner integral.  Its image in
$L^2(\Omega)$ is the same integral taken in $L^2$.  Consequently the
$L^2$ and $H^m$ definitions of $\mathcal L_1(\mathbb D)$ have the same
functions, representing measures, and variation norms, and
\[
 \|f\|_{H^m(\Omega)}\lesssim\|f\|_{\mathcal L_1(\mathbb D)}.
\]
\end{lemma}

\begin{proof}
For $|\alpha|\le m$, the weak spatial derivative of a ridge is
\[
 D_x^\alpha T(\omega,b)
 =\frac{k!}{(k-|\alpha|)!}\omega^\alpha
   (\omega\cdot x-b)_+^{k-|\alpha|},
\]
where exponent zero means $\mathbf1_{\{\omega\cdot x>b\}}$.
These derivatives are uniformly bounded on the bounded domain and compact
parameter set.  When $(\omega_n,b_n)\to(\omega,b)$ they converge pointwise
outside the hyperplane $\omega\cdot x=b$, a null set.  Dominated convergence
gives convergence in $L^2$ for every derivative, hence continuity in $H^m$.

The bounded continuous map $T$ is Bochner integrable against any finite
measure.  The continuous inclusion $H^m(\Omega)\hookrightarrow L^2(\Omega)$
commutes with this integral and is injective.  Thus an $L^2$ representation
automatically represents the same function in $H^m$, and conversely.
Finally, $\|\int T\,d\mu\|_{H^m}\le\sup_\Theta\|T\|_{H^m}\|\mu\|_{\rm TV}$;
taking the infimum proves the norm bound.

We also record that a ball of representing measures of bounded total
variation has compact image in $H^m$. By weak-star compactness, any
sequence of such measures has a weak-star convergent subsequence on
the compact metric space $\Theta$. Approximate the continuous map $T$
uniformly by a finite sum $\sum_i\rho_i(\theta)T(\theta_i)$, where
$\rho_i$ is a continuous partition of unity subordinate to a sufficiently
fine finite cover of $\Theta$. The integrals of this finite-rank map
converge in $H^m$, and the uniform approximation error is bounded by
the common measure mass times the approximation tolerance. Letting that
tolerance tend to zero proves norm convergence of the original integrals.
The limiting measure has no larger total variation. In particular,
variation balls are closed in $H^m$.
\end{proof}

For comparison with the symmetric bias interval in
\cite{LiWangSIMA2026}, put $R_\Omega:=\sup_{x\in\Omega}|x|$ and choose
$B$ with $R_\Omega<B<\min\{-c,d_0\}$.  The companion dictionary with
biases in $[-B,B]$ is contained in $\mathbb D$.  Its representing measures
can be extended by zero to $\Theta$ without increasing their total
variation, and Lemma~\ref{lem:variation-topology} upgrades those
representations from $L^2$ to $H^m$.

\subsection{Scaling and wavelet reduction}\label{sec:scaling}

The upper argument first reduces the critical Besov ball to localized smooth
atoms and then discretizes their ridge representations.  A finite-dimensional
polynomial term may appear in this reduction.  For comparison, define
\[
 \widehat\Sigma_{n,M}
 :=\left\{P+g:P\in\Pi_k,\ g\in\Sigma_{n,M}\right\}.
\]
Write $\widehat\sigma_{n,m}$ and $\widehat E_{n,m}$ for the corresponding
errors.  Since $\Sigma_{n,M}\subset\widehat\Sigma_{n,M}$,
\[
 \sigma_{n,m}(f;M)\ge\widehat\sigma_{n,m}(f;M),
 \qquad
 E_{n,m}(B_X;M)\ge\widehat E_{n,m}(B_X;M).
\]

\paragraph{Polynomial corrections.}

The lower proofs use no free polynomial. For the upper construction,
the following lemma realizes a polynomial correction by finitely many
dictionary atoms with a controlled coefficient mass.

\begin{lemma}[Finite realization of $\Pi_k$ by elements of $\mathbb D$]\label{lem:poly-realization}
Assume that $\Omega$ is bounded.  There exist an integer $J=J(d,k,\Omega)$, fixed parameters $(\omega_j,b_j)\in\Sph^{d-1}\times[-c,c]$, $1\le j\le J$, and a bounded linear map
\[
T:\Pi_k\longrightarrow\R^J,
\qquad
T(P)=(c_1(P),\ldots,c_J(P)),
\]
such that for every $P\in\Pi_k$,
\[
P(x)=\sum_{j=1}^{J}c_j(P)(\omega_j\cdot x-b_j)_+^k,
\qquad x\in\Omega,
\]
and
\[
\sum_{j=1}^{J}|c_j(P)|\lesssim \|P\|_{H^m(\Omega)}.
\]
\end{lemma}

\begin{proof}
Fix a unit vector $\omega$.  Since $\Omega$ is bounded, there exists $B_\Omega>0$ such that
$|\omega\cdot x|\le B_\Omega$ for all $x\in\Omega$ and all $\omega\in\Sph^{d-1}$.  Choose $k+1$ distinct numbers
\[
 b_0,\ldots,b_k\in
 \left(c,\inf_{x\in\Omega,\,|\omega|=1}\omega\cdot x\right).
\]
Then $\omega\cdot x-b_j>0$ for every $x\in\Omega$, hence
\[
(\omega\cdot x-b_j)_+^k=(\omega\cdot x-b_j)^k
=\sum_{\ell=0}^{k}\binom{k}{\ell}(-b_j)^{k-\ell}(\omega\cdot x)^\ell.
\]
For fixed $\omega$, the coefficient matrix in the variables
$1,(\omega\cdot x),\ldots,(\omega\cdot x)^k$ is a Vandermonde matrix in the distinct numbers $-b_j$ and is therefore invertible.  Consequently every ridge monomial $(\omega\cdot x)^\ell$, $0\le\ell\le k$, is a linear combination of the $k+1$ always-active $\operatorname{ReLU}^k$ ridges above.

For a fixed degree $\ell$, the homogeneous polynomial space $\mathcal H_\ell$ is naturally identified with the symmetric tensor space $\operatorname{Sym}^\ell(\R^d)$.  Rank-one symmetric tensors $\omega^{\otimes\ell}$ span this finite-dimensional space: if a symmetric tensor $A$ were orthogonal to every $\omega^{\otimes\ell}$, then the homogeneous polynomial $x\mapsto\langle A,x^{\otimes\ell}\rangle$ would vanish for every $x$ and hence $A=0$.  Thus one may choose finitely many directions $\omega_{\ell,1},\ldots,\omega_{\ell,N_\ell}$ such that $(\omega_{\ell,r}\cdot x)^\ell$ spans $\mathcal H_\ell$.  Taking the union over $0\le\ell\le k$ and replacing every ridge monomial by the Vandermonde representation above yields a fixed finite family of $\operatorname{ReLU}^k$ ridges spanning $\Pi_k|_\Omega$.

Let $S:\R^J\to\Pi_k$ be the resulting surjective linear synthesis map.  Since both spaces are finite-dimensional, $S$ has a linear right inverse $T$, and all norms are equivalent.  Hence
\[
\|T(P)\|_{\ell^1}\lesssim \|P\|_{H^m(\Omega)}.
\]
This proves both assertions.
\end{proof}

The preceding realization controls the polynomial coefficient cost.
To use it uniformly in the width, we also need a bound for the size of
the polynomial appearing in an augmented approximation.

\begin{proposition}[Transfer of a stable augmented upper estimate to the dictionary class]\label{prop:pure-transfer}
Suppose that for some $\gamma>0$, $\kappa\ge0$ and every $f$ in a quasi-normed unit ball $B_X$ there exists, for every integer $N\ge1$, an approximant
\[
\widetilde g_N=P_N+\sum_{r=1}^{N}a_r
\sigma_k(\omega_r\cdot x-b_r),
\qquad P_N\in\Pi_k,\quad
(\omega_r,b_r)\in\Sph^{d-1}\times[-c,c],
\]
satisfying
\[
\|f-\widetilde g_N\|_{H^m(\Omega)}\lesssim  N^{-\gamma}(1+\log(2+N))^\kappa,
\qquad
\sum_{r=1}^{N}|a_r|\le M_0,
\]
and assume $X\hookrightarrow H^m(\Omega)$.  Then there is a fixed $\widetilde M_0<\infty$ such that, for every integer $n\ge1$,
\[
E_{n,m}(B_X;\widetilde M_0)
\lesssim n^{-\gamma}(1+\log(2+n))^\kappa.
\]
\end{proposition}

\begin{proof}
The parameter set $\Sph^{d-1}\times[-c,c]$ is compact and $0\le m\le k$, hence
\[
\sup_{(\omega,b)\in\mathbb S^{d-1}\times[-c,c]}
\|\sigma_k(\omega\cdot x-b)\|_{H^m(\Omega)}<\infty.
\]
For $f\in B_X$ the embedding gives $\|f\|_{H^m}\lesssim 1$.  Thus
\[
\|P_N\|_{H^m}
\le \|f-\widetilde g_N\|_{H^m}+\|f\|_{H^m}
 +\left\|\sum_{r=1}^{N}a_r
 \sigma_k(\omega_r\cdot x-b_r)\right\|_{H^m}
\lesssim (1+M_0),
\]
uniformly in $N$ and $f$.  Lemma~\ref{lem:poly-realization} represents $P_N$ by at most $J$ dictionary atoms with total coefficient mass bounded by a fixed constant $M_\Pi$.  Hence $\widetilde g_N$ is exactly a dictionary network of width at most $N+J$ and budget at most $M_0+M_\Pi=:\widetilde M_0$.

Given $n>2J$, choose $N=n-J$.  Then $N\simeq n$ and $\log(2+N)\simeq\log(2+n)$, which yields the asserted rate.  The finitely many $n\le2J$ are absorbed into the constant.
\end{proof}

\paragraph{Critical bump scaling and Besov estimates.}

The lower constructions and the fine-scale remainder in the upper estimate
use the same critical rescaling.  We record its Besov and Sobolev behavior
once here so that every later width calculation reduces to counting bumps
and parameter cells.

Fix $x_0\in\Omega$ and $r_0>0$ with $B(x_0,4r_0)\Subset\Omega$.  Let $0\ne\psi\in C_c^\infty(B(0,1))$ and for dyadic $0<h<r_0$ define
\begin{align}\label{eq:psih-def}
    \psi_h(x):=h^k\psi\left(\frac{x-x_0}{h}\right).
\end{align}

\begin{lemma}[Uniform critical Besov norm]\label{lem:single-besov}
For $s_0=k+d/p$, $0<p<1$ and $0<q\le1$,
\[
\sup_{0<h<r_0}\|\psi_h\|_{B_{p,q}^{s_0}(\Omega)}<\infty.
\]
\end{lemma}

\begin{proof}
The general rescaled-bump argument is developed in
\cite[Section~4.1]{LiWangSIMA2026}. Its stated two-sided lemma uses
a specially chosen moment-cancelling bump. Here arbitrary smooth
bumps, including radial ones, are needed, so we give the upper bound
in this more general form.

Write $h=2^{-J}$ and choose an integer $L>s_0$. Let
$\Delta_jF=K_j*F$, where $K_j(x)=2^{jd}K(2^jx)$ for $j\ge1$,
$K$ is Schwartz, and every moment of $K$ vanishes. If $j\le J$,
the direct convolution formula and compact support of $\psi$ give
\[
 |\Delta_j\psi_h(x)|
 \lesssim h^{k+d}2^{jd}
       (1+2^j|x-x_0|)^{-M}.
\]
For $j=0$ the same bound holds with the fixed low-frequency kernel.
Choose $M>d/p$ and integrate its $p$th power. This yields
\[
 2^{js_0}\|\Delta_j\psi_h\|_p
 \lesssim (2^jh)^{k+d},\qquad 0\le j\le J.
\]
For $j>J$ set $\lambda=2^jh>1$ and use
\[
 \Delta_j\psi_h(x_0+hy)
 =h^k\int K(z)\psi(y-z/\lambda)\,dz.
\]
Taylor expansion of $\psi$ through degree $L-1$ and the moments of
$K$ give a remainder bounded by $C\lambda^{-L}|z|^L$.
If $y$ is outside a fixed enlargement of $\operatorname{supp}\psi$,
a nonzero derivative in that remainder requires
$|z|\gtrsim\lambda|y|$. Schwartz decay of $K$ therefore gives, for
arbitrarily large $M$,
\[
 |\Delta_j\psi_h(x_0+hy)|
 \lesssim h^k\lambda^{-L}(1+|y|)^{-M}.
\]
Taking the $L^p$ quasi-norm and multiplying by $2^{js_0}$ gives
\[
 2^{js_0}\|\Delta_j\psi_h\|_p
 \lesssim(2^jh)^{s_0-L},\qquad j>J.
\]
The two geometric sequences have a uniformly bounded $\ell^q$
quasi-norm because $k+d>0$ and $L>s_0$. Restriction to $\Omega$
proves the lemma.
\end{proof}

The source norm is now controlled independently of $h$. The next
identity determines the size of the same bump in the error norm.

\begin{lemma}[$H^m$ scaling]\label{lem:single-Hm}
For $0\le m\le k$,
\[
\|\psi_h\|_{H^m(\Omega)}\simeq h^{k-m+d/2}=h^{\tau_m}.
\]
\end{lemma}

\begin{proof}
For each multi-index $\alpha$,
\[
D^\alpha \psi_h(x)
=h^{k-|\alpha|}(D^\alpha\psi)\left(\frac{x-x_0}{h}\right),
\]
whence the change of variables $y=(x-x_0)/h$ gives
\[
\|D^\alpha \psi_h\|_{L^2(\Omega)}
=h^{k-|\alpha|+d/2}\|D^\alpha\psi\|_{L^2(\Omega)}.
\]
Summing over $|\alpha|\le m$ yields the upper bound $Ch^{k-m+d/2}$, because $0<h\le1$ and the smallest exponent occurs at $|\alpha|=m$.  Since a nonzero compactly supported smooth function cannot have all derivatives of order $m$ equal to zero unless it is a polynomial of degree $<m$, hence identically zero, at least one $D^\alpha\psi$ with $|\alpha|=m$ is nonzero when $m\ge1$; for $m=0$ use $\psi\ne0$.  The corresponding term gives the matching lower bound.
\end{proof}

The angular obstruction uses one bump. The direction--bias obstruction
uses a row of bumps, for which the following normalization preserves
admissibility independently of their number and signs.

\begin{lemma}[Uniform Besov norm of an $\ell^p$-normalized row]\label{lem:row-besov}
For $k\in\N_+$, let $\left\{x_j\right\}_{j=1}^k\subset\Omega$ be arbitrary and
\begin{align}\label{eq:psihj-def}
    \psi_{h,j}(x):=h^k\psi\left(\frac{x-x_j}{h}\right)
\end{align}
for which are supported in a fixed compact subset of $\Omega$.  For arbitrary signs $\eps_j\in\{-1,1\}$ let
\[
F_{h,k}:=k^{-1/p}\sum_{j=1}^{K}\eps_j \psi_{h,j}.
\]
Then
\[
\|F_{h,k}\|_{B_{p,q}^{s_0}(\Omega)}\lesssim 1.
\]
\end{lemma}
\begin{proof}
    For each Littlewood--Paley block $\Delta_i$, translation invariance
and $p$-subadditivity in $L^p(\R^d)$ give
\[
 \|\Delta_i F_{h,k}\|_{L^p(\R^d)}^p
 \le k^{-1}\sum_{j=1}^k\|\Delta_i \psi_{h,j}\|_{L^p(\R^d)}^p
 =\|\Delta_i \psi_h\|_{L^p(\R^d)}^p,
\]
taking the weighted outer $\ell^q$ quasi-norm and applying the
whole-space estimate in Lemma~\ref{lem:single-besov}
proves the assertion.
\end{proof}

\begin{lemma}[$H^m$ size of a disjoint row]\label{lem:row-Hm}
Assume in addition that the supports of the $u_{h,j}$ are pairwise disjoint. Then
\[
\|F_{h,k}\|_{H^m(\Omega)}\simeq k^{1/2-1/p}h^{\tau_m}.
\]
\end{lemma}

\begin{proof}
All derivatives of the bumps have the same supports as the bumps, hence are pairwise disjoint.  Therefore, for every $|\alpha|\le m$,
\[
\left\|D^\alpha F_{h,K}\right\|_2^2
=K^{-2/p}\sum_{j=1}^{K}\|D^\alpha u_{h,j}\|_2^2.
\]
Using the scaling from Lemma~\ref{lem:single-Hm} and summing over $|\alpha|\le m$ gives
\[
\|F_{h,K}\|_{H^m}^2\simeq K^{1-2/p}h^{2\tau_m}.
\]
Taking square roots proves the result.
\end{proof}

\paragraph{Stable fixed-budget upper bound.}\label{sec:upper}

The next representation provides the coefficient bounds used in the
upper construction. The local approximation of one compactly supported
atom and the allocation argument are developed in
Section~\ref{sec:upper-proof}.

Put
\[
\gamma_m:=\frac12+\frac{\beta_m}{d},
\qquad
r_m:=k+1-m=\beta_m+\frac12.
\]
Recall
\[
A_m=\frac{\tau_m}{d-1},
\qquad
G_{p,m}=\gamma_m+\frac{1/p-1}{d}.
\]
Let
\[
\kappa_\ast:=\left(A_m+\frac12-\frac1q\right)_+.
\]

\paragraph{The critical wavelet coefficient model.}

\begin{lemma}[Compactly supported critical wavelet representation]
\label{lem:critical-wavelet-model}
Fix $0<p\le1$, $0<q\le1$, and a finite differentiability order $L$.
Every $f\in B_{p,q}^{k+d/p}(\Omega)$ admits a compactly supported extension
$F$ and an expansion
\[
 F=\sum_{j\ge0}\sum_{\nu\in\Lambda_j}\lambda_{j,\nu}a_{j,\nu}
 \quad\text{in }H^m(\R^d),
\]
where
\[
 h_j=2^{-j},\qquad
 a_{j,\nu}(x)=\psi_{e(j,\nu)}
 \left(\frac{x-x_{j,\nu}}{h_j}\right),
\]
the templates belong to one fixed finite family in $C_c^L(\R^d)$,
each $\Lambda_j$ is finite, and
\[
 b_j:=h_j^{-k}\|(\lambda_{j,\nu})_\nu\|_{\ell^p},
 \qquad
 \|(b_j)\|_{\ell^q}
 \lesssim \|f\|_{B_{p,q}^{k+d/p}(\Omega)}.
\]
Moreover,
\begin{align}
 \sum_{j,\nu}h_j^{-k}|\lambda_{j,\nu}|
 &\le \sum_jb_j\le\|(b_j)\|_{\ell^q},
 \label{eq:critical-global-l1}\\
 \left\|\sum_{j,\nu}\eta_{j,\nu}a_{j,\nu}\right\|_{H^m(\R^d)}^2
 &\lesssim \sum_{j,\nu}h_j^{2(\tau_m-k)}|\eta_{j,\nu}|^2
 \label{eq:critical-bessel}
\end{align}
for every finitely supported scalar family $(\eta_{j,\nu})$.
The same wavelet system can be chosen to satisfy, for
$u\in\{p,p_\ast,1\}$, the synthesis bound
\begin{equation}
 \left\|\sum_{j,\nu}\eta_{j,\nu}a_{j,\nu}\right\|_{B_{u,q}^{k+d/u}(\R^d)}
 \lesssim
 \left\|\left(h_j^{-k}\|(\eta_{j,\nu})_\nu\|_{\ell^u}\right)_j
 \right\|_{\ell^q}.
 \label{eq:critical-besov-synthesis}
\end{equation}
When the right-hand side is finite, the series converges in this Besov
space and in $H^m(\R^d)$.  Restriction to $\Omega$ preserves the bound.
Thus $a_{j,\nu}$ has unit amplitude and spatial scale $h_j$, while the
factor $h_j^{-k}=2^{jk}$ occurs in the coefficient norm and budget.
\end{lemma}

\begin{proof}
The universal extension theorem for Besov spaces on bounded Lipschitz domains
\cite{Rychkov1999}
gives a bounded operator
$\mathcal E:B_{p,q}^{k+d/p}(\Omega)\to
B_{p,q}^{k+d/p}(\R^d)$.  Multiplying $\mathcal Ef$ by a fixed smooth cutoff
which equals one near $\overline\Omega$ gives a compactly supported extension
without changing the norm by more than a fixed factor.

Choose a compactly supported orthonormal tensor-product wavelet basis with
regularity at least $L$ and sufficiently many derivatives and vanishing
moments for the Besov characterizations at all three indices
$u\in\{p,p_\ast,1\}$ and the $H^m=B_{2,2}^m$ characterization.
Such a simultaneous choice is possible because only finitely many
smoothness and integrability indices are involved; use the Daubechies
construction \cite[Chapters~6--7]{Daubechies1992} and the wavelet
characterization \cite[Section~3.1]{Triebel2006}.
The inhomogeneous level $j=0$ includes the scaling functions.

If $\alpha_{j,\nu}$ are the
usual coefficients of the $L^2$-normalized wavelets
\[
 \psi_{j,\nu}(x)=h_j^{-d/2}\psi_{e(j,\nu)}
 \left(\frac{x-x_{j,\nu}}{h_j}\right),
\]
the wavelet characterization of Besov spaces reads
\[
 \left\|
 \left(2^{j(k+d/2)}\|(\alpha_{j,\nu})_\nu\|_{\ell^p}
 \right)_{j\ge0}\right\|_{\ell^q}
 \simeq \|F\|_{B_{p,q}^{k+d/p}}.
\]
Only wavelets whose supports meet the fixed support of $F$ occur; therefore,
at every scale the index set $\Lambda_j$ is finite and all rescaled functions
come from the finite collection of tensor-product mother wavelets and scaling
functions.  Set $\lambda_{j,\nu}=h_j^{-d/2}\alpha_{j,\nu}$.  Then
$\alpha_{j,\nu}\psi_{j,\nu}=\lambda_{j,\nu}a_{j,\nu}$ and the displayed Besov
equivalence gives the claimed bound for $(b_j)$.  Since $p,q\le1$,
$\ell^p\hookrightarrow\ell^1$ and $\ell^q\hookrightarrow\ell^1$, proving
\eqref{eq:critical-global-l1}.

For the Sobolev estimate, the $H^m$ wavelet norm equivalence follows from the same
wavelet theorem applied with smoothness $m$ and assigns the factor $h_j^{-m}$
to an $L^2$-normalized wavelet.  The coefficient of that wavelet in
$\eta_{j,\nu}a_{j,\nu}$ is $\eta_{j,\nu}h_j^{d/2}$; its squared Sobolev weight is
therefore
\[
 |\eta_{j,\nu}|^2h_j^dh_j^{-2m}
 =|\eta_{j,\nu}|^2h_j^{2(\tau_m-k)}.
\]
Summing the wavelet norm equivalence proves \eqref{eq:critical-bessel}.
For the original coefficients its right-hand side is bounded by
\[
 \sum_jh_j^{2(d/2-m)}\|\lambda_j\|_{\ell^2}^2
 \le\sum_jh_j^{2\tau_m}b_j^2<\infty,
\]
since $p\le1$, $\tau_m>0$, and $(b_j)\in\ell^q\subset\ell^2$.
Thus the expansion converges in $H^m$ to the same extension $F$.

For the synthesis assertion, the $L^2$-normalized coefficients are
$h_j^{d/2}\eta_{j,\nu}$.  The Besov synthesis theorem assigns them the
weight
\[
 h_j^{-(k+d/u)-d/2+d/u}h_j^{d/2}=h_j^{-k},
\]
which proves \eqref{eq:critical-besov-synthesis}.  Finite coefficient
families are dense in the indicated sequence spaces because $u,q<\infty$.
Applying \eqref{eq:critical-bessel} with $p$ replaced by $u\le1$
proves convergence in $H^m$ as well.
\end{proof}

The differentiability order $L$ may be prescribed as large as needed
before choosing the basis. No single compactly supported basis is assumed
to have infinite regularity.  In particular, we choose it after fixing the
finite angular cubature order used below. On $\Omega$ we may discard
atoms whose supports do not meet $\Omega$, without changing the sum or
increasing any of the coefficient bounds.

Lemma~\ref{lem:critical-wavelet-model} reduces the global problem to a
finite family of localized unit-amplitude templates.  The companion
scaled-atom variation estimate bounds the coefficient cost of representing
each template at scale $h$.

The next estimate is \cite[Lemma~3.6]{LiWangSIMA2026} in the
unit-amplitude normalization used here. We use the same symmetric
subdictionary and the topology identification above.

\begin{lemma}[Localized-atom variation estimate
{\cite[Lemma~3.6]{LiWangSIMA2026}}]
\label{lem:localized-atom-variation}
Let $R_0>k+d+2$ be an integer. Suppose that $A$ is supported in a
fixed ball and ranges over a bounded subset of $C_c^{R_0+1}(\R^d)$.
For $0<h\le1$ and $x_0\in\R^d$,
\[
 \left\|A\left(\frac{\cdot-x_0}{h}\right)\bigg|_\Omega
 \right\|_{\mathcal L_1(\mathbb D)}\lesssim h^{-k}.
\]
The constant is independent of $h,x_0,A$.
\end{lemma}

The cited lemma is stated for smooth templates with a bound involving
only derivatives through order $R_0$. Its use for $C_c^{R_0+1}$
templates follows by mollification: smooth approximants have a common
support bound and common $C^{R_0}$ bound, and converge in $H^m$
after each fixed dilation. Closedness of the variation ball, proved
in Lemma~\ref{lem:variation-topology}, passes the cited estimate to the
limit. This is the only extension of that statement used here.

Summing atom costs at critical smoothness is already the content of
the companion embedding theorem. We quote it, rather than reproduce
its Littlewood--Paley argument.

\begin{proposition}[Critical Besov--variation embedding
{\cite[Theorem~1.1]{LiWangSIMA2026}}]
\label{prop:critical-embedding}
For $0<p\le1$ and $0<q\le1$,
\[
 \|f\|_{\mathcal L_1(\mathbb D)}
 \lesssim\|f\|_{B_{p,q}^{k+d/p}(\Omega)}.
\]
The representing integral is an $H^m(\Omega)$ Bochner integral.
\end{proposition}

Indeed, the companion theorem applies on every bounded Lipschitz
domain with the symmetric bias interval fixed above.
Lemma~\ref{lem:variation-topology} identifies its $L^2$ representation
with the $H^m$ representation, without increasing the measure mass.
The next two lemmas concern discretization, which is not supplied by
that embedding.

\begin{lemma}[Stratified high-order cubature]
\label{lem:stratified-cubature}
Let $Y$ be a compact smooth manifold of dimension $r\ge1$ with a fixed
smooth measure, $H$ a Hilbert space, and $s\ge1$ an integer.
For every $G\in C^s(Y;H)$ and integer $K\ge1$ there is an unbiased
random formula
\[
 \mathcal Q_KG=\sum_{i=1}^{N_K}w_iG(y_i),\qquad N_K\lesssim K,
\]
whose weights satisfy, for every realization,
\[
 \sum_i|w_i|\lesssim1,\qquad\sum_i|w_i|^2\lesssim K^{-1},
\]
and
\[
 \mathbb E\left\|\int_YG\,dy-\mathcal Q_KG\right\|_H^2
 \lesssim K^{-1-2s/r}\|G\|_{C^s(Y;H)}^2.
\]
The $C^s$ norm includes every chart derivative of order at most $s$.
\end{lemma}

\begin{proof}
Choose finitely many coordinate charts and a smooth partition of unity
with support strictly inside their coordinate domains. In each chart,
the corresponding integral becomes the integral over a fixed cube of
$G$ composed with the chart, multiplied by a fixed smooth density.
Extend this product by zero outside the support of the partition
function. Its $C^s$ norm is bounded by $\|G\|_{C^s(Y;H)}$.

Partition each cube into $O(K)$ congruent subcubes of side
$\delta\simeq K^{-1/r}$. On a subcube $C$, interpolate the product by
a degree-$(s-1)$ polynomial $I_CG$ at a fixed unisolvent set of nodes.
Affine scaling preserves the Lebesgue constant, and Taylor's formula
gives
\[
 \|G-I_CG\|_{L^\infty(C;H)}
 \lesssim \delta^s\|G\|_{C^s(Y;H)},
\]
where the fixed density is included in $G$ in this local formula.
Integrate the interpolant exactly and add
$|C|(G-I_CG)(Y_C)$, with $Y_C$ uniform on $C$ and all cells sampled
independently. The resulting formula is unbiased. On each cell its
fixed number of weights, after incorporating the smooth density, have
absolute value at most $C|C|$. Summation over the finite chart family
therefore gives both pathwise weight bounds.

The errors are independent and centered, so their squared Hilbert
norms add in expectation. Their total variance is bounded by
\[
 \sum_C|C|^2\|G-I_CG\|_{L^\infty(C;H)}^2
 \lesssim K^{-1-2s/r}\|G\|_{C^s(Y;H)}^2.
\]
Nodes where the extended density is zero have zero weight and can
be omitted. This completes the construction on the manifold.
\end{proof}

To apply the general variation-space approximation theorem, we need
the precise parameter smoothness of a ridge in the error norm.
The following lemma verifies it, including the case $m=k$.

\begin{lemma}[Local interpolation of the ridge dictionary]
\label{lem:dictionary-interpolation}
Let $\Theta=\mathbb S^{d-1}\times[-c,c]$, so $\dim\Theta=d$.
For every $N\ge1$ there are nodes $\theta_i\in\Theta$ and
coefficient functions $\lambda_i$ with at most $CN$ nonzero nodes such that
\[
 I_N(\omega,b)(x):=\sum_i\lambda_i(\omega,b)
 \sigma_k(\omega_i\cdot x-b_i),
 \qquad \theta_i=(\omega_i,b_i),
 \qquad
 \sup_{(\omega,b)\in\Theta}\sum_i|\lambda_i(\omega,b)|\lesssim 1,
\]
and
\[
 \sup_{(\omega,b)\in\Theta}
 \|\sigma_k(\omega\cdot x-b)-I_N(\omega,b)\|_{H^m(\Omega)}
 \lesssim  N^{-\beta_m/d}.
\]
\end{lemma}

\begin{proof}
Put $q_0=k-m$.  We first establish the precise parameter regularity used by
the interpolant.  In one fixed chart write
\[
 \ell_\theta(x)=\omega(\theta')\cdot x-b,
 \qquad T(\theta)=\sigma_k(\ell_\theta)\in H^m(\Omega).
\]
For $|\alpha|\le m$ and $|\gamma|\le q_0$, the chain and product rules show
that $\partial_\theta^\gamma D_x^\alpha T(\theta)$ is a finite sum of terms
\[
 C_{\alpha,\gamma,j}(\theta,x)
 (\ell_\theta(x))_+^{k-|\alpha|-j},
 \qquad 0\le j\le |\gamma|,
\]
with uniformly bounded smooth coefficients; when the exponent is zero the
last factor is interpreted as $\mathbf 1_{\{\ell_\theta>0\}}$.  The displayed formula
is obtained first off the hyperplane $\ell_\theta=0$ and then in the weak
sense. The corresponding parameter difference quotients converge in
$L^2$ by dominated convergence and the bounded scalar difference-quotient
estimates for truncated powers. No boundary measure occurs at these
orders because $j\le k-|\alpha|$.

Let $|\theta-\theta'|\le1$.  Since $\Omega$ is bounded,
$|\ell_\theta(x)-\ell_{\theta'}(x)|\lesssim |\theta-\theta'|$.  The set on which
the two affine functions have different signs is contained in
\[
 \{x\in\Omega:|\ell_\theta(x)|\lesssim |\theta-\theta'|\}.
\]
After an orthogonal change of coordinates, Fubini's theorem bounds the
measure of this slab by $C|\theta-\theta'|$, uniformly in the chart.  Apply
this observation to the terms in the displayed derivative formula with
$|\gamma|=q_0$.  On the slab their
$L^2$ difference is $O(|\theta-\theta'|^{1/2})$; off the slab the ordinary
mean-value theorem gives the smaller bound $O(|\theta-\theta'|)$.  Therefore
\[
 \|\partial_\theta^\gamma T(\theta)
       -\partial_\theta^\gamma T(\theta')\|_{H^m(\Omega)}
 \lesssim |\theta-\theta'|^{1/2},
 \qquad |\gamma|=q_0.
\]
The same calculation with fewer derivatives proves their continuity.  Thus
$T$ is $C^{q_0,1/2}$ as an $H^m(\Omega)$-valued map, with a uniform norm on
all of the finitely many charts.  The same estimate also covers $q_0=0$.

Subdivide each chart into shape-regular cells of diameter
$\delta\simeq N^{-1/d}$ and bounded overlap.  On a cell use a fixed
unisolvent interpolation formula of degree $q_0$ (piecewise constant when
$q_0=0$).  Its number of nodes and Lebesgue constant depend only on
$d,q_0$.  Taylor's formula with the preceding H\"older remainder, followed by the
uniform Lebesgue bound, gives
\[
 \|T-I_NT\|_{H^m(\Omega)}
 \lesssim \delta^{q_0+1/2}=\delta^{\beta_m}
\]
on every cell.  A fixed partition of the chart overlaps changes neither the
order nor the uniform sum $\sum_i|\lambda_i|$.  There are $O(N)$ nodes in
total, and $\delta\simeq N^{-1/d}$ gives the asserted estimate.
\end{proof}

The parameter smoothness just established permits a direct application
of the existing approximation theorem for smoothly parameterized
dictionaries. We record its consequence in our Sobolev norm.

\begin{theorem}[Stable approximation of the variation ball]
\label{thm:variation-discretization}
For every $f\in\mathcal L_1(\mathbb D)$ and $n\ge1$ there is
$g_n\in\Sigma_{n,C\|f\|_{\mathcal L_1(\mathbb D)}}$ such that
\[
 \|f-g_n\|_{H^m(\Omega)}
 \lesssim \|f\|_{\mathcal L_1(\mathbb D)}
 n^{-1/2-\beta_m/d}
 =\|f\|_{\mathcal L_1(\mathbb D)}n^{-\gamma_m}.
\]
\end{theorem}

This is \cite[Theorem~2]{SiegelXu2024} with the type-$2$ space
$H^m(\Omega)$, parameter dimension $d$, and smoothness
$\beta_m=k-m+1/2$. The $C^{k-m,1/2}$ parameter regularity was verified
in Lemma~\ref{lem:dictionary-interpolation}, and
Lemma~\ref{lem:variation-topology} identifies the closed atomic ball
with the measure-defined ball. The cited theorem controls the total
coefficient mass independently of the width. Homogeneity gives the
statement for general $f$, and the empty network handles $f=0$.
We use this existing approximation theorem without reproducing its
sampling proof.

\begin{corollary}[$p=1$ variation baseline]
\label{cor:p1-variation-baseline}
For $0<q\le1$, $f\in B_{1,q}^{k+d}(\Omega)$, and $n\ge1$,
\[
 \sigma_{n,m}\left(f;C\|f\|_{B_{1,q}^{k+d}(\Omega)}\right)
 \lesssim \|f\|_{B_{1,q}^{k+d}(\Omega)}n^{-\gamma_m}.
\]
\end{corollary}

\begin{proof}
The critical embedding, Proposition~\ref{prop:critical-embedding}, gives
$\|f\|_{\mathcal L_1(\mathbb D)}\lesssim \|f\|_{B_{1,q}^{k+d}(\Omega)}$.  Substitute this estimate into
Theorem~\ref{thm:variation-discretization}.
\end{proof}

\section{Proof of the upper bound}\label{sec:upper-proof}

The preceding section reduced the common part of the upper argument to a
stable variation-space approximation.  We now quantify the approximation of
one localized wavelet atom, allocate a fixed number of neurons across its
ridge parameters and across the dyadic scales, and finally treat the three
possible positions of $p$ relative to $p_\ast$.  This division keeps
the analytic preparation and the case-dependent proof logically separate.

\subsection{Localized atoms and resource allocation}

\paragraph{The localized unit-amplitude atom oracle.}

\begin{lemma}[Weighted one-dimensional bias cubature]
\label{lem:weighted-bias-cubature}
Let $H$ be a Hilbert space, $q_0\in\N_0$,
$\beta=q_0+1/2$, and $r=q_0+1$.  Suppose that
$z:\R\to\R$ is measurable and satisfies
\[
 |z(s)|\lesssim (1+|s|)^{-\nu},
 \qquad \nu>\beta+1.
\]
Let $J\subset\R$ be an interval.  Assume that $V_h\in C^{q_0}(J;H)$,
$\sup_{s\in J}\|V_h(s)\|_H\lesssim1$, and
\[
 [V_h^{(q_0)}]_{C^{0,1/2}(J;H)}\lesssim h^\beta,
 \qquad 0<h\le1.
\]
In particular, on every bounded interval $I\subset J$, fixed-node
degree-$q_0$ interpolation has a uniformly bounded Lebesgue constant and
remainder
\[
 \|V_h-I_IV_h\|_{L^\infty(I;H)}
 \lesssim \min\{1,h^\beta|I|^\beta\},
 \qquad 0<h\le1.
\]
Then, for every integer $L\ge1$, there is an
unbiased random formula $\mathcal Q_{h,L}$, involving at most $C(q_0,\nu)L$
values of $V_h$ at points of $J$, such that
\begin{align*}
 \mathbb E\left\|
  \int_Jz(s)V_h(s)\,ds-\mathcal Q_{h,L}
 \right\|_H^2&\lesssim  h^{2\beta}L^{-2r},\\
 \sum|w_\ell(\mathcal Q_{h,L})|
 &\lesssim \int_J|z(s)|(1+|s|)^\beta\,ds\lesssim1.
\end{align*}
The constants are independent of $h,L$, and $J$.
\end{lemma}

\begin{proof}
Put $\delta:=(\nu-\beta-1)/(\beta+1)>0$ and
\[
 J_L:=\left\lfloor\frac{\log_2 L}{\delta}\right\rfloor,
 \qquad S:=2^{J_L+1}\simeq L^{1/\delta}.
\]
Divide $[-1,1]$ into $L$ intervals.  For each
$0\le\ell\le J_L$, divide each of the shells
$[2^\ell,2^{\ell+1}]$ and $[-2^{\ell+1},-2^\ell]$ into
\[
 M_\ell:=\left\lceil L2^{-\ell\delta}\right\rceil
\]
equal intervals, and intersect all cells with $J$.  The total number of
nonempty cells is $O(L)$, since
$\sum_{\ell\le J_L}M_\ell\lesssim L+J_L+1\lesssim L$.
On a cell $I$ in the $\ell$th shell, $|I|\lesssim2^\ell/M_\ell$ and
$\mu_I:=\int_I|z|\lesssim2^{\ell(1-\nu)}/M_\ell$.  Hence
\begin{equation}
 \mu_I|I|^\beta
 \lesssim2^{\ell(\beta+1-\nu)}M_\ell^{-(\beta+1)}.
 \label{eq:weighted-cell}
\end{equation}

On each cell integrate its interpolant $P_I:=I_IV_h$ exactly against
$z(s)ds$, and add the independent correction
\[
 \mu_I\operatorname{sgn}z(Y_I)(V_h-P_I)(Y_I),
 \qquad Y_I\sim |z(s)|\,ds/\mu_I.
\]
Cells of zero mass contribute zero.  This formula is unbiased, has a
fixed number of nodes, and has coefficient mass at most $C\mu_I$.
By \eqref{eq:weighted-cell}, the sum of its variances over one shell is
\[
 Ch^{2\beta}2^{2\ell(\beta+1-\nu)}M_\ell^{-2\beta-1}
 \lesssim h^{2\beta}L^{-2\beta-1}2^{-\ell\delta}.
\]
The central interval $[-1,1]$ contributes at most
$Ch^{2\beta}L^{-2\beta-1}$.  Summing the shell estimates therefore bounds
the variance on $J\cap[-S,S]$ by the same order.

We treat the two tails with polynomial extrapolation and importance
sampling.  Consider a nonempty positive tail $T:=J\cap[S,\infty)$ and
write $u:=\inf T\ge S$.  If $|T|\le u$, a single cell formula as above
has variance at most
\[
 Ch^{2\beta}\left(u^{-\nu}|T|^{\beta+1}\right)^2
 \lesssim h^{2\beta}u^{2(\beta+1-\nu)}.
\]
If $|T|>u$, interpolate at $q_0+1$ fixed relative nodes inside $[u,2u]$,
and call the resulting polynomial $P_T$.  Its Lagrange coefficients obey
\[
 \sum_i|\ell_i(s)|\lesssim (s/u)^{q_0}\le(s/u)^\beta,
 \qquad s\in T.
\]
Taylor's formula with the assumed H\"older bound, first on $[u,2u]$ and
then from the same base point to $s$, gives
\[
 \|V_h(s)-P_T(s)\|_H\lesssim h^\beta s^\beta.
\]
Indeed, interpolation reproduces the degree-$q_0$ Taylor polynomial;
its remainder at the interpolation nodes is $O(h^\beta u^\beta)$,
whose extrapolation is bounded by
$Ch^\beta u^\beta(s/u)^{q_0}\lesssim h^\beta s^\beta$.
Set
\[
 Z_T:=\int_T|z(s)|(s/u)^\beta\,ds\lesssim u^{1-\nu}.
\]
Integrate $zP_T$ exactly, and add the correction
\[
 Z_T\operatorname{sgn}z(Y)(u/Y)^\beta
       (V_h-P_T)(Y),
 \qquad Y\sim |z(s)|(s/u)^\beta\,ds/Z_T,
\]
with zero contribution if $Z_T=0$.  It is unbiased and has variance at
most
\[
 CZ_T^2h^{2\beta}u^{2\beta}
 \lesssim h^{2\beta}u^{2(\beta+1-\nu)}
 \lesssim h^{2\beta}S^{2(\beta+1-\nu)}
 \lesssim h^{2\beta}L^{-2\beta-2}.
\]
Its pathwise coefficient mass is at most $CZ_T$, by the Lagrange bound
and $(s/u)^\beta\ge1$ on $T$, and
$Z_T\le\int_T|z(s)|(1+|s|)^\beta ds$ since $u\ge1$.
The short-tail case has the same coefficient bound.  Reflection treats
the negative tail, with independent randomness.

Adding the central and tail variances gives
$Ch^{2\beta}L^{-2\beta-1}=Ch^{2\beta}L^{-2r}$.
There are $O(L)$ evaluations in total, and summing the coefficient bounds
gives the stated weighted $L^1$ budget, which is finite because
$\nu>\beta+1$.
\end{proof}

Lemma~\ref{lem:weighted-bias-cubature} applies to the Fourier--Peano
ridge profile once its decay and bias regularity are established.
The next lemma verifies these hypotheses and controls the angular
derivatives needed for the complementary spherical discretization.

\begin{lemma}[Localized ridge-profile estimates]
\label{lem:localized-profile-estimates}
Fix an integer $s_{\rm ang}\ge1$ and choose the wavelet templates with
sufficiently large finite regularity depending on $d,k,m,s_{\rm ang}$.  Let
$a_{h,x_0}(x)=\psi((x-x_0)/h)$, where $\psi$ ranges over the finite
template family of Lemma~\ref{lem:critical-wavelet-model}, $0<h\le1$,
and the atom's support meets $\Omega$.  There are a
polynomial $P_h\in\Pi_k$ and $H^m$-valued profiles
$\Psi_{h,\omega}$ such that
\[
 a_{h,x_0}=P_h+\int_{\mathbb S^{d-1}}\Psi_{h,\omega}\,\mathrm d\omega,
 \qquad
 \Psi_{h,\omega}=\int_c^{d_0}\zeta_{h,\omega}(b)
 \sigma_k(\omega\cdot x-b)\,\mathrm db.
\]
For this prescribed angular order,
\begin{equation}
 \max_{|\alpha|\le s_{\rm ang}}
 \|\nabla_\omega^\alpha\Psi_{h,\omega}\|_{H^m(\Omega)}
 \lesssim h^{\beta_m-k-s_{\rm ang}}.
 \label{eq:profile-angular}
\end{equation}
For every $L\ge1$ there is an unbiased random sum $R_{h,\omega,L}$ of at
most $C L$ atoms of $\mathbb D$ such that
\begin{equation}
 \sum_\ell|a_\ell(R_{h,\omega,L})|\lesssim h^{-k},
 \qquad
 \mathbb E\|\Psi_{h,\omega}-R_{h,\omega,L}\|_{H^m}^2
 \lesssim h^{2(\beta_m-k)}L^{-2r_m}.
 \label{eq:profile-bias}
\end{equation}
Finally, $P_h$ is exactly representable by a fixed number of atoms with
coefficient sum bounded by $Ch^{-k}$, uniformly in $x_0$ and the template.
\end{lemma}

\begin{proof}
\textbf{Step 1: exact Fourier--Peano representation and profile decay.}
For a unit-scale template $A$, write its Radon transform as
\[
 (\mathcal RA)(t,\omega)=\int_{\omega^\perp}A(t\omega+y)\,dy.
\]
The Fourier-slice identity and the polar inversion used in
Lemma~\ref{lem:localized-atom-variation} show that the $(k+1)$st derivative
of the one-dimensional profile is
\begin{equation}
 F_{h,\omega}^{(k+1)}(b)
 =h^{-k-1}z_\omega
   \left(\frac{b-\omega\cdot x_0}{h}\right),
 \qquad
 \widehat z_\omega(r)=c_d(ir)^{k+1}|r|^{d-1}\widehat A(r\omega).
 \label{eq:profile-scaling}
\end{equation}
Taylor's formula with integral remainder on $[-c,c]$, followed by
integration in $\omega$, therefore gives
\[
 a_{h,x_0}=P_h+\int_{\Sph^{d-1}}\Psi_{h,\omega}\,d\omega,
 \qquad
 \zeta_{h,\omega}(b)=\frac1{k!}F_{h,\omega}^{(k+1)}(b).
\]

We record the decay needed below.  If $d$ is odd, $|r|^{d-1}=r^{d-1}$,
and Fourier inversion in \eqref{eq:profile-scaling} makes $z_\omega$ a constant multiple of
$\partial_t^{k+d}\mathcal RA(t,\omega)$; it has the required finite regularity and is compactly
supported in a fixed interval.  If $d$ is even, the same formula is a
Hilbert transform of $\partial_t^{k+d}\mathcal RA$.  The latter function is
compactly supported and has vanishing moments of orders $0,\ldots,k+d-1$.
For $|s|$ larger than twice its support radius, expand
\[
 \frac1{s-t}=\sum_{j=0}^{k+d-1}\frac{t^j}{s^{j+1}}
 +\frac{t^{k+d}}{s^{k+d}(s-t)}.
\]
The moment terms vanish, so the last term gives
\begin{equation}
 |z_\omega(s)|\lesssim (1+|s|)^{-k-d-1}.       \label{eq:profile-decay}
\end{equation}
Differentiation in a smooth angular chart only differentiates the Radon
transform and preserves compact support and the same moment cancellations.
Consequently, for every integer $0\le a\le s_{\rm ang}$,
\begin{equation}
 \sup_\omega|\nabla_\omega^az_\omega(s)|
 \lesssim (1+|s|)^{-k-d-1}.                  \label{eq:profile-angular-decay}
\end{equation}
The finite template family and its chosen wavelet regularity make the
constants in \eqref{eq:profile-decay}--\eqref{eq:profile-angular-decay}
uniform.  Integrating those bounds gives uniform $L^1$ control of the
profile and its prescribed angular derivatives.

\textbf{Step 2: the angular $H^m$ estimate.}
Define the unit-scale profile by
\[
 V_\omega(t)=c_d\int_\R e^{irt}|r|^{d-1}\widehat A(r\omega)\,dr,
 \qquad F_{h,\omega}(b)=
 V_\omega\left(\frac{b-\omega\cdot x_0}{h}\right).
\]
Fourier decay from the prescribed finite regularity of $A$ and
one-dimensional Plancherel imply uniform $L^2(\R)$ bounds for
$\partial_t^\ell\nabla_\omega^aV_\omega$ whenever
$\ell+a\le m+s_{\rm ang}$. We choose the template regularity large
enough for these finitely many bounds and for Step~1.

We also need control of the Taylor polynomial removed by the
truncated Peano formula. For odd $d$, $V_\omega$ is a derivative
of the compactly supported Radon transform. For even $d$, it is
a constant multiple of the Hilbert transform of
$\partial_t^{d-1}\mathcal RA$. Repeating the moment expansion in
Step~1, now with $d-1+\ell$ derivatives, gives
\[
 |\partial_t^\ell\nabla_\omega^aV_\omega(t)|
 \lesssim(1+|t|)^{-d-\ell}
\]
outside a fixed interval, for all the finite orders needed here.
Angular differentiation does not change the support radius or the
vanishing moments of the differentiated Radon profile.

Since $\omega\cdot x\in(c,d_0)$ on $\Omega$, Taylor's formula gives
the exact identity
\[
 \Psi_{h,\omega}(x)=
 V_\omega\left(\frac{\omega\cdot(x-x_0)}h\right)
 -\sum_{\ell=0}^k\frac{F_{h,\omega}^{(\ell)}(c)}{\ell!}
                    (\omega\cdot x-c)^\ell.
\]
For the first term, differentiation in $x$ and in an angular chart
produces bounded polynomial factors in $x-x_0$, together with at
most $h^{-|\alpha|-a}$ for $|\alpha|\le m$ and $a$ angular
derivatives. All relevant centres lie in a fixed bounded set.
Fubini in coordinates parallel and orthogonal to $\omega$, followed
by the uniform one-dimensional $L^2$ bounds, consequently gives
\[
 \left\|\nabla_\omega^aD_x^\alpha
 V_\omega\left(\frac{\omega\cdot(x-x_0)}h\right)
 \right\|_{L^2(\Omega)}
 \lesssim h^{1/2-|\alpha|-a}.
\]

The support-intersection condition implies
$\operatorname{dist}(x_0,\Omega)\lesssim h$. For sufficiently small
$h$, the fixed exterior bias $c$ is thus separated from
$\omega\cdot x_0$ by a positive constant. Every term in an
$a$th angular derivative of $F_{h,\omega}^{(\ell)}(c)$ is bounded
by
\[
 h^{-\ell-v}
 \left|\partial_t^{\ell+v}\nabla_\omega^{a'}V_\omega
       \left(\frac{c-\omega\cdot x_0}{h}\right)\right|
 \lesssim h^d,\qquad v+a'\le a.
\]
The same bound therefore holds for the $H^m$ norm of the
differentiated Taylor polynomial. The finitely many larger dyadic
scales have uniformly bounded profiles and centres and can be
absorbed into the constant. Summing over spatial derivatives gives
\[
 \|\nabla_\omega^a\Psi_{h,\omega}\|_{H^m(\Omega)}
 \lesssim h^{1/2-m-a}=h^{\beta_m-k-a},
 \qquad a\le s_{\rm ang},
\]
which proves \eqref{eq:profile-angular}, including all lower chart
derivatives required by Lemma~\ref{lem:stratified-cubature}.

\textbf{Step 3: high-order unbiased discretization of the bias.}
Put $q_0=k-m$, $\beta=q_0+1/2=\beta_m$, and
\[
 J_{h,\omega}=\left[\frac{c-\omega\cdot x_0}{h},
                  \frac{d_0-\omega\cdot x_0}{h}\right],
 \qquad
 V_{h,\omega}(s)(x)=\sigma_k(\omega\cdot x-\omega\cdot x_0-hs).
\]
After the substitution $b=\omega\cdot x_0+hs$,
\eqref{eq:profile-scaling} becomes
\begin{equation}
 \Psi_{h,\omega}
 =h^{-k}\int_{J_{h,\omega}}z_\omega(s)V_{h,\omega}(s)\,ds/k!.
 \label{eq:normalized-bias-profile}
\end{equation}
The map $V_{h,\omega}$ is uniformly bounded in $H^m(\Omega)$.  On an
interval $I$ of length $\delta$, interpolate it in $s$ with degree $q_0$.
Exactly as in the slab calculation in
Lemma~\ref{lem:dictionary-interpolation}, the $q_0$th parameter derivative
is $1/2$-H\"older: the moving hyperplane sweeps a slab of thickness
$h|s-s'|$.  Including the $q_0$ chain-rule factors of $h$ gives
\[
 [\partial_s^{q_0}V_{h,\omega}]_{C^{0,1/2}(J_{h,\omega};H^m)}
 \lesssim h^{q_0+1/2}.
\]
Consequently,
\begin{equation}
 \|V_{h,\omega}-I_IV_{h,\omega}\|_{L^\infty(I;H^m)}
 \lesssim \min\{1,h^{q_0+1/2}\delta^{q_0+1/2}\}.
 \label{eq:bias-interpolation-remainder}
\end{equation}
Since $k+d+1>\beta_m+1$,
\eqref{eq:profile-decay}, \eqref{eq:normalized-bias-profile}, and
\eqref{eq:bias-interpolation-remainder} allow us to apply
Lemma~\ref{lem:weighted-bias-cubature}.  Multiplication by $h^{-k}/k!$
gives a formula with $CL$ atoms,
\[
 \mathbb E\|\Psi_{h,\omega}-R_{h,\omega,L}\|_{H^m}^2
 \lesssim h^{-2k}h^{2\beta_m}L^{-2(k+1-m)},
\]
which is the mean-square estimate in \eqref{eq:profile-bias}.

The bias formula is unbiased cell by cell, including the two tails.
Its pathwise coefficient mass is at most
\[
 Ch^{-k}\int_\R|z_\omega(s)|(1+|s|)^{\beta_m}\,ds
 \lesssim h^{-k},
\]
because \eqref{eq:profile-decay} has decay exponent
$k+d+1>\beta_m+1$.

\textbf{Step 4: the polynomial term.}
The exact profile identity \eqref{eq:profile-scaling} gives
\[
 \int_{\Sph^{d-1}}\int_c^{d_0}|\zeta_{h,\omega}(b)|\,db\,d\omega
 \lesssim h^{-k}.
\]
Every dictionary atom has uniformly bounded $H^m(\Omega)$ norm.  The Peano
identity therefore implies
$\|P_h\|_{H^m(\Omega)}\le
\|a_{h,x_0}\|_{H^m}+Ch^{-k}\lesssim h^{-k}$.
Lemma~\ref{lem:poly-realization} represents $P_h$ exactly by a fixed number
of dictionary atoms with coefficient mass $Ch^{-k}$, completing the
representation and budget estimates.
\end{proof}

Lemma~\ref{lem:localized-profile-estimates} controls angular
differentiation and bias discretization separately.  Balancing the numbers
of angular cells and bias samples combines those estimates into one atom
approximant with the angular width factor $h^{-(d-1)}$.

\begin{proposition}[Localized unit-amplitude atom oracle]
\label{prop:localized-atom-oracle}
For an integer $s\ge1$ set
\[
 \rho_s:=r_m\frac{\frac12+s/(d-1)}{r_m+s/(d-1)}.
\]
Then $\rho_s\uparrow r_m$.  For every $z\ge1$ there is an unbiased random
finite dictionary network $Q_{h,z}$ such that
\begin{align*}
 \operatorname{width}(Q_{h,z})&\lesssim h^{-(d-1)}z,\\
 \sum_\ell|a_\ell(Q_{h,z})|&\lesssim h^{-k},\\
 \mathbb E\|a_{h,x_0}-Q_{h,z}\|_{H^m}^2
 &\lesssim h^{2(\tau_m-k)}z^{-2\rho_s}.
\end{align*}
Here $a_{h,x_0}$ is an atom from
Lemma~\ref{lem:localized-profile-estimates}, with angular order $s$
and the corresponding fixed finite template regularity.
All constants are uniform over this family and its admissible centres.
\end{proposition}

\begin{proof}
Write $r=r_m$ and use
Lemma~\ref{lem:localized-profile-estimates}.  Apply
Lemma~\ref{lem:stratified-cubature} on $\mathbb S^{d-1}$ with $K$ angular
cells.  By \eqref{eq:profile-angular}, the angular mean-square error is
\[
 C_sh^{2(\beta_m-k-s)}K^{-1-2s/(d-1)}.
\]
Replace every profile evaluation occurring in the cubature formula by an
independent copy of the $L$-point bias formula \eqref{eq:profile-bias}.  The squared angular
weights sum to $O(K^{-1})$, so independence gives the additional error
\[
 C_sh^{2(\beta_m-k)}K^{-1}L^{-2r}.
\]
The polynomial remainder is inserted exactly.  Thus
\begin{equation}
 \mathbb E\|a_{h,x_0}-Q\|_{H^m}^2
 \lesssim \left[
 h^{2(\beta_m-k-s)}K^{-1-2s/(d-1)}
 +h^{2(\beta_m-k)}K^{-1}L^{-2r}\right].
 \label{eq:oracle-mse}
\end{equation}
The pathwise coefficient budget is $O(h^{-k})$ because the cubature weights
have bounded $\ell^1$ mass and every bias formula has that budget.
The width is $O(KL)$.

Put $T=h^{-(d-1)}z$ and choose integers comparable to
\[
 K=h^{-(d-1)}z^{r/(r+s/(d-1))},
 \qquad
 L=z^{(s/(d-1))/(r+s/(d-1))}.
\]
Then $KL\simeq h^{-(d-1)}z$.  The two terms in \eqref{eq:oracle-mse} are comparable, and a
direct substitution gives
\[
 h^{2(\beta_m-k-s)}K^{-1-2s/(d-1)}
 =h^{2(\beta_m-k)+d-1}
 z^{-r(1+2s/(d-1))/(r+s/(d-1))}
 =h^{2(\tau_m-k)}z^{-2\rho_s},
\]
because $2(\beta_m-k)+d-1=2(\tau_m-k)$.  Rounding $K,L$ changes only the constant.
Unbiasedness follows successively from the angular and bias formulas.
Finally, the displayed definition shows $\rho_s\to r_m$ as $s\to\infty$.
\end{proof}

The atom approximation improves as more neurons are assigned to it.
The following sequence estimate distributes that local cost among
coefficients and controls the coordinates that are omitted.

\begin{lemma}[Sequence resource allocation]\label{lem:resource}
Let $0<p_0<2$ and
\[
a_0:=\frac1{p_0}-\frac12>0.
\]
Assume $\rho>a_0$.  If $\lambda=(\lambda_r)\in\ell^{p_0}$ and $b=\|\lambda\|_{\ell^{p_0}}$, then for every $T\ge1$ there exist numbers
\[
z_r\in\{0\}\cup[1,\infty)
\]
such that
\[
\sum_{z_r>0}z_r\lesssim T
\]
and
\[
\sum_{z_r>0}|\lambda_r|^2z_r^{-2\rho}
+\sum_{z_r=0}|\lambda_r|^2
\lesssim b^2T^{-2a_0}.
\]
\end{lemma}

\begin{proof}
Rearrange the coefficients so that $\lambda_r^*$ is nonincreasing.  Since
\[
 r(\lambda_r^*)^{p_0}
 \le\sum_{i=1}^r(\lambda_i^*)^{p_0}
 \le b^{p_0},
\]
we have
\[
\lambda_r^*\le br^{-1/p_0}=br^{-a_0-1/2}.
\]
Choose
\[
\frac{a_0}{\rho}<\eta<1,
\]
put $M=\max\{1,\lfloor c_\eta T\rfloor\}$, and define
\[
z_r=(M/r)^\eta\quad(1\le r\le M),
\qquad
z_r=0\quad(r>M).
\]
Since $\eta<1$,
\[
\sum_{r=1}^{M}z_r
=M^\eta\sum_{r=1}^{M}r^{-\eta}
\lesssim  M.
\]
Choosing $c_\eta$ small makes this at most $CT$ with the desired fixed
constant.  When $1\le T<2/c_\eta$, the choice $M=1$ changes all subsequent
estimates only by a constant because $T\simeq1$; below we may therefore
assume $M\simeq T$.

For the active coordinates,
\[
\sum_{r\le M}(\lambda_r^*)^2z_r^{-2\rho}
\le b^2M^{-2\eta\rho}
\sum_{r\le M}r^{-2a_0-1+2\eta\rho}.
\]
Because $\eta\rho>a_0$, the exponent in the last sum is greater than $-1$, hence
\[
\sum_{r\le M}r^{-2a_0-1+2\eta\rho}
\lesssim M^{-2a_0+2\eta\rho}.
\]
Thus the active contribution is at most $Cb^2M^{-2a_0}$.  The omitted tail satisfies
\[
\sum_{r>M}(\lambda_r^*)^2
\le b^2\sum_{r>M}r^{-2a_0-1}
\lesssim b^2M^{-2a_0}.
\]
Since $M\simeq T$, the result follows.
\end{proof}

Combining the sequence allocation with independent atom approximants
gives a bound on one scale. Its proof retains the omitted-coordinate
energy, which is needed when different scales are later combined.

\begin{lemma}[Stable estimate on one wavelet scale]\label{lem:one-scale}
Let
\[
f_j=\sum_\nu \lambda_{j,\nu}a_{j,\nu}.
\]
Suppose $0<p_0<2$, put $a_0=1/p_0-1/2$, and choose $s$ in Proposition~\ref{prop:localized-atom-oracle} so that $\rho_s>a_0$.  For every integer $N\ge0$ there exists a deterministic dictionary network $G_{j,N}$ satisfying
\[
\operatorname{width}(G_{j,N})\lesssim N,
\qquad
\sum_\ell|a_\ell(G_{j,N})|
\lesssim h_j^{-k}\sum_\nu|\lambda_{j,\nu}|,
\]
and
\[
\|f_j-G_{j,N}\|_{H^m}
\lesssim h_j^{-k}\|(\lambda_{j,\nu})\|_{\ell^{p_0}}
 h_j^{\tau_m}(1+Nh_j^{d-1})^{-a_0}.
\]
\end{lemma}

\begin{proof}
If $Nh_j^{d-1}<1$, take $G_{j,N}=0$.  The same-scale case of the Bessel estimate gives
\[
\|f_j\|_{H^m}^2
\lesssim h_j^{2(\tau_m-k)}\|(\lambda_{j,\nu})\|_{\ell^2}^2
\lesssim h_j^{2(\tau_m-k)}\|(\lambda_{j,\nu})\|_{\ell^{p_0}}^2,
\]
and $1+Nh_j^{d-1}\simeq1$.

Assume now $Nh_j^{d-1}\ge1$ and put
\[
T\simeq Nh_j^{d-1}.
\]
Apply Lemma~\ref{lem:resource} to the coefficient vector
$\lambda_j=(\lambda_{j,\nu})_\nu$ with $\rho=\rho_s$, obtaining costs
$z_\nu$.  For each active coordinate use an independent oracle
$Q_{j,\nu,z_\nu}$ from Proposition~\ref{prop:localized-atom-oracle}; omit
the coordinates with $z_\nu=0$.  Define
\[
D_{j,N}:=\sum_{z_\nu=0}\lambda_{j,\nu}a_{j,\nu},
\qquad
\widetilde G_{j,N}:=\sum_{z_\nu>0}\lambda_{j,\nu}Q_{j,\nu,z_\nu}.
\]
Unbiasedness gives $\mathbb E\widetilde G_{j,N}=f_j-D_{j,N}$.  Independence and zero mean eliminate cross terms among active coordinates, while the same-scale Bessel estimate controls the deterministic omitted part.  Therefore
\[
\mathbb E\|f_j-D_{j,N}-\widetilde G_{j,N}\|_{H^m}^2
+\|D_{j,N}\|_{H^m}^2
\]
\[
\lesssim h_j^{2(\tau_m-k)}
\left(
\sum_{z_\nu>0}|\lambda_{j,\nu}|^2z_\nu^{-2\rho_s}
+\sum_{z_\nu=0}|\lambda_{j,\nu}|^2
\right)
\lesssim \|\lambda_j\|_{\ell^{p_0}}^2h_j^{2(\tau_m-k)}T^{-2a_0}.
\]
By Proposition~\ref{prop:localized-atom-oracle} and the resource bound,
\[
\operatorname{width}(\widetilde G_{j,N})
\lesssim h_j^{-(d-1)}\sum_{z_\nu>0}z_\nu
\lesssim h_j^{-(d-1)}T\lesssim N.
\]
The pathwise output budget is
\[
\sum_\ell|a_\ell(\widetilde G_{j,N})|
\lesssim h_j^{-k}\sum_{z_\nu>0}|\lambda_{j,\nu}|
\lesssim h_j^{-k}\sum_\nu|\lambda_{j,\nu}|.
\]
Choose a realization whose squared random error does not exceed its expectation and then add the deterministic omitted error by the triangle inequality.  This gives the stated deterministic $G_{j,N}$.
\end{proof}

The endpoint estimate will be assembled directly from these scale
bounds. For larger $p$, we also need a decomposition between the
endpoint and the $p=1$ class that preserves the actual coefficient mass.

\begin{lemma}[Critical disjoint-coordinate splitting]\label{lem:splitting}
Let $0<p_0<p<1$ and
\[
\theta:=\frac{1/p-1}{1/p_0-1}\in(0,1).
\]
For every $\lambda\in\ell^p$ and $0<t\le1$, there is a decomposition into disjoint coordinate sets
\[
\lambda=\lambda^{(0)}+\lambda^{(1)}
\]
such that
\[
\|\lambda^{(0)}\|_{\ell^1}+t\|\lambda^{(1)}\|_{\ell^{p_0}}
\lesssim t^\theta\|\lambda\|_{\ell^p},
\]
and
\[
\|\lambda^{(0)}\|_{\ell^1}+\|\lambda^{(1)}\|_{\ell^1}
=\|\lambda\|_{\ell^1}.
\]
\end{lemma}

\begin{proof}
Let $b=\|\lambda\|_{\ell^p}$ and rearrange $|\lambda_r|$ in nonincreasing order.  Then
\[
\lambda_r^*\le br^{-1/p}.
\]
Choose
\[
K\simeq t^{-1/(1/p_0-1)}.
\]
Let $\lambda^{(1)}$ contain the first $K$ coordinates and $\lambda^{(0)}$ the remaining tail. Then
\[
\|\lambda^{(0)}\|_{\ell^1}
\le b\sum_{r>k}r^{-1/p}
\lesssim bk^{1-1/p},
\]
and, since $p_0<p$,
\[
\|\lambda^{(1)}\|_{\ell^{p_0}}^{p_0}
\le b^{p_0}\sum_{r\le K}r^{-p_0/p}
\lesssim b^{p_0}K^{1-p_0/p},
\]
so
\[
\|\lambda^{(1)}\|_{\ell^{p_0}}
\lesssim bK^{1/p_0-1/p}.
\]
By the choice of $K$,
\[
K^{1-1/p}\simeq t^\theta,
\qquad
tK^{1/p_0-1/p}\simeq t^\theta.
\]
This proves the first estimate.  The supports are disjoint by construction, hence the $\ell^1$ masses add exactly.
\end{proof}

\subsection{The three parameter regimes}

The estimates above are uniform in the scale and in the wavelet position.
Their summation changes at the threshold $p_\ast$, which leads to the
following three-case theorem.

\begin{theorem}[Stable fixed-$\ell^1$-budget upper bound]\label{thm:pure-upper}
Let $d\ge3$, $0<p<1$, $0<q\le1$, and $0\le m\le k$.  There is $C>0$ such that for every
$f\in B_{p,q}^{k+d/p}(\Omega)$ and every $n\ge1$,
\[
\sigma_{n,m}
\left(f;C\|f\|_{B_{p,q}^{k+d/p}(\Omega)}\right)
\lesssim \|f\|_{B_{p,q}^{k+d/p}(\Omega)}
 n^{-\Gamma_{p,m}}L_n^{\kappa_{p,q,m}},
\]
where
\[
\kappa_{p,q,m}=
\begin{cases}
0,&0<p<p_\ast,\\[0.3em]
\left(A_m+\frac12-\frac1q\right)_+,&p=p_\ast,\\[0.6em]
\displaystyle
\frac{1/p-1}{1/p_\ast-1}
\left(A_m+\frac12-\frac1q\right)_+,
&p_\ast<p<1.
\end{cases}
\]
The output budget is independent of $n$.
\end{theorem}

\begin{proof}
Set $B:=\|f\|_{B_{p,q}^{k+d/p}(\Omega)}$. If $B=0$, take the
empty network. Otherwise use Lemma~\ref{lem:critical-wavelet-model},
choosing its finite wavelet regularity after the angular order needed
in the relevant case. We first allow width at most a fixed multiple
of $n$, then rescale the integer width at the end. The three regimes
are proved separately.

\textbf{Case 1: $p=p_\ast$.}
Then
\[
\frac1p-\frac12=A_m,
\qquad
\tau_m=(d-1)A_m.
\]
Because $d\ge3$, $d-1\ge2$, and
\[
r_m-A_m
=\beta_m\left(1-\frac1{d-1}\right)>0.
\]
Choose $s$ so that $\rho_s>A_m$, and apply Lemma~\ref{lem:one-scale} with $p_0=p_\ast$ and $a_0=A_m$.

Let
\[
J=\left\lfloor\frac{\log_2 n}{d-1}\right\rfloor,
\qquad
L=J+1,
\qquad
\delta=(A_m+1/2)^{-1},
\qquad
S=\sum_{j=0}^{J}b_j^\delta.
\]
If $S=0$, all coarse scales vanish.  Otherwise allocate
\[
N_j=\left\lfloor\frac{n b_j^\delta}{S}\right\rfloor,
\qquad 0\le j\le J,
\]
where, if the one-scale construction uses at most $C_0N_j$ atoms, $n$ in
this definition is first replaced by $\lfloor n/C_0\rfloor$.  Then
$\sum_jC_0N_j\le n$.  The random precursor in the proof of
Lemma~\ref{lem:one-scale} may be taken independently across scales.  Since
$\tau_m=(d-1)A_m$, its mean-square error on scale $j$ is bounded by
\[
Cb_j^2(1+N_j)^{-2A_m}.
\]
Independence removes cross terms among the centered random errors, and the all-scale Bessel estimate controls the deterministic omitted coordinates.  Thus there is a deterministic realization satisfying
\[
\left\|\sum_{j=0}^{J}(f_j-G_{j,N_j})\right\|_{H^m}^2
\lesssim \sum_{j=0}^{J}b_j^2(1+N_j)^{-2A_m}.
\]
Since $1+\lfloor x\rfloor\ge c(1+x)$,
\[
b_j^2(1+N_j)^{-2A_m}
\lesssim n^{-2A_m}S^{2A_m}b_j^{2-2A_m\delta}.
\]
The choice $\delta=(A_m+1/2)^{-1}$ gives
$2-2A_m\delta=\delta$, hence
\[
\left(\sum_{j=0}^{J}b_j^2(1+N_j)^{-2A_m}\right)^{1/2}
\lesssim n^{-A_m}S^{A_m+1/2}.
\]
But
\[
S^{A_m+1/2}=\|(b_j)_{j\le J}\|_{\ell^\delta},
\]
and on a vector of length $L$,
\[
\|b\|_{\ell^\delta}
\le L^{(1/\delta-1/q)_+}\|b\|_{\ell^q}
=L^{\kappa_\ast}\|b\|_{\ell^q}.
\]
Therefore the coarse-scale error is at most
\[
CBn^{-A_m}L_n^{\kappa_\ast}.
\]
For the fine tail $j>J$, Lemma~\ref{lem:critical-wavelet-model} and $\ell^p\hookrightarrow\ell^2$ give
\[
\left\|\sum_{j>J}f_j\right\|_{H^m}^2
\lesssim \sum_{j>J}2^{-2j(d-1)A_m}b_j^2
\lesssim n^{-2A_m}B^2.
\]
The output budget is bounded by
\[
 C\sum_{j,\nu}h_j^{-k}|\lambda_{j,\nu}|\lesssim B.
\]
This proves the endpoint.

\textbf{Case 2: $0<p<p_\ast$.}
Here $1/p-1/2>A_m$.  Since $r_m>A_m$, choose
\[
A_m<\alpha<\min\left\{\frac1p-\frac12,r_m\right\}
\]
and define $p_0$ by $1/p_0-1/2=\alpha$.  Then $p<p_0<2$, so
\[
h_j^{-k}\|\lambda_j\|_{\ell^{p_0}}\le b_j,
\]
and choose $s$ with $\rho_s>\alpha$.  Keep
\[
J=\left\lfloor\frac{\log_2 n}{d-1}\right\rfloor.
\]
For $0\le j\le J$ set
\[
N_j=\left\lceil c_0n^{A_m/\alpha}
2^{j(d-1)(\alpha-A_m)/\alpha}\right\rceil,
\]
with $c_0$ small.  This is a geometric progression and $\sum_{j\le J}N_j\lesssim n$ because $2^{J(d-1)}\simeq n$.  Lemma~\ref{lem:one-scale} gives
\[
\|f_j-G_{j,N_j}\|_{H^m}
\lesssim b_jh_j^{\tau_m}(N_jh_j^{d-1})^{-\alpha}
=b_jN_j^{-\alpha}2^{j(d-1)(\alpha-A_m)}
\lesssim b_jn^{-A_m}.
\]
Summing over $j$ and using $\sum_jb_j\lesssim B$ gives coarse-scale error $CBn^{-A_m}$.  The fine tail is the same as in Case 1.  The budget is again $CB$.  Since $G_{p,m}\ge A_m$ precisely when $p\le p_\ast$, this is the desired saturated rate with no logarithmic loss.

\textbf{Case 3: $p_\ast<p<1$.}
Set
\[
\theta:=\frac{1/p-1}{1/p_\ast-1}\in(0,1).
\]
For a parameter $0<t\le1$, apply Lemma~\ref{lem:splitting} to each scale coefficient vector with $p_0=p_\ast$:
\[
 \lambda_j=\lambda_j^{(0)}+\lambda_j^{(1)},
\]
with disjoint supports and
\[
 h_j^{-k}\left(\|\lambda_j^{(0)}\|_{\ell^1}+t\|\lambda_j^{(1)}\|_{\ell^{p_\ast}}\right)
\lesssim t^\theta b_j.
\]
Let $f^{(0)}$ and $f^{(1)}$ be the corresponding wavelet sums.  Taking the outer $\ell^q$ norm gives
\[
\|f^{(0)}\|_{B_{1,q}^{k+d}}
\lesssim t^\theta B,
\qquad
\|f^{(1)}\|_{B_{p_\ast,q}^{k+d/p_\ast}}
\lesssim t^{\theta-1}B.
\]
Approximate $f^{(0)}$ with half the width using Corollary~\ref{cor:p1-variation-baseline}, and $f^{(1)}$ with the other half using the endpoint estimate just proved.  Then
\[
\sigma_{n,m}(f;CB)
\lesssim Bn^{-\gamma_m}t^\theta
+CBn^{-A_m}L_n^{\kappa_\ast}t^{\theta-1}.
\]
Choose
\[
t=n^{-(A_m-\gamma_m)}L_n^{\kappa_\ast},
\]
with $t=1$ for the finitely many $n$ where this exceeds $1$.  Direct calculation gives
\[
A_m-\gamma_m=\frac{\beta_m}{d(d-1)},
\]
and, since $1/p_\ast-1=\beta_m/(d-1)$,
\[
\theta(A_m-\gamma_m)=\frac{1/p-1}{d}.
\]
Thus either term is bounded by
\[
CBn^{-\gamma_m-\theta(A_m-\gamma_m)}L_n^{\theta\kappa_\ast}
=Bn^{-G_{p,m}}L_n^{\theta\kappa_\ast}.
\]
This is precisely the displayed logarithmic exponent.  For the budget, the $f^{(0)}$ approximation uses $O(t^\theta B)\le O(B)$ by Corollary~\ref{cor:p1-variation-baseline}.  For $f^{(1)}$, the endpoint construction uses the actual $\ell^1$ mass of active coordinates, and the disjoint splitting gives
\[
\sum_{j,\nu}h_j^{-k}|\lambda_{j,\nu}^{(1)}|
\le\sum_{j,\nu}h_j^{-k}|\lambda_{j,\nu}|
\lesssim B.
\]
Thus the combined output budget is $O(B)$ uniformly in $t$ and $n$.

In every case the construction has width at most $C_w n$, with
$C_w$ fixed, and coefficient mass at most $C_b B$. Apply the
construction with $\lfloor n/C_w\rfloor$ when $n\ge2C_w$.
Powers of this integer and their logarithmic factors are comparable
to the stated powers of $n$. For $n<2C_w$ take the zero network and
use $\|f\|_{H^m}\lesssim B$, which follows from the critical wavelet
estimate. Enlarging only the error constant covers these finitely
many widths and proves the theorem.
\end{proof}

\section{Lemmas for the lower bounds}\label{sec:lower-lemmas}

The lower bound has two independent components.  The first uses only the
$(d-1)$-dimensional direction sphere and is valid without a coefficient
budget.  The second resolves both directions and biases; it is weaker in the
angularly saturated range but becomes sharp when $p>p_\ast$.

\paragraph{Angular obstruction for the dictionary class.}

We use the radial comparison theorem of Konovalov, Leviatan, and Maiorov
\cite[Theorem~4]{KLM2008}. Write $B^d=B(0,1)$,
$\mathcal P_t$ for polynomials of degree at most $t$, and
$E_2(F,\mathcal A;B)=\inf_{G\in\mathcal A}\|F-G\|_{L^2(B)}$.
Here $\mathcal R_{N,2}$ consists of sums of $N$ univariate ridges whose
profiles belong to $L^2(-1,1)$ on the unit ball.
It supplies constants $\bar c>0$ and
$\bar c_0\in\N$, depending only on $d$, such that for every radial
$F\in L^2(B^d)$ and every $s\in\N$,
\[
\bar c\,E_2(F,\mathcal P_{\bar c_0s};B^d)
\le E_2(F,\mathcal R_{s^{d-1},2};B^d),
\]
We also use its one-dimensional Remez inequality
\cite[Lemma~10]{KLM2008}. These two external results are used without
reproof. The following localization and scaling consequences are needed
in the present Sobolev setting, so we include their derivation.

\begin{lemma}[Scaled scalar radial-ridge obstruction]\label{lem:scaled-klm}
Let $B(x_0,4R)\Subset\Omega$ and let $0\ne\varphi\in\mathcal S(\R^d)$ be radial.  For $a\in\R$ define
\[
\varphi_h^{(a)}(x):=h^a\varphi\left(\frac{x-x_0}{h}\right).
\]
There exist $c_0,c_1,h_0>0$ such that, if $0<h<h_0$ and $N\le c_1h^{-(d-1)}$, then
\[
\inf_{G\in\mathcal R_N}
\|\varphi_h^{(a)}-G\|_{L^2(B(x_0,R))}
\ge c_0h^{a+d/2},
\]
where $\mathcal R_N$ is the class of sums of at most $N$ univariate ridge profiles $q_j(\omega_j\cdot x+b_j)$ with
$q_j\in L^2_{\rm loc}(\R)$.  The constants are independent of $a$.
\end{lemma}

\begin{proof}
Choose $C_\varphi>0$ so large that
\[
 \|\varphi\|_{L^2(\R^d\setminus B(0,C_\varphi))}
 \le \varepsilon\|\varphi\|_{L^2(\R^d)},
\]
where $\varepsilon>0$ will be fixed below.  On the unit ball set
\[
F_\delta(y):=A\varphi(y/\delta),
\]
where $A$ is arbitrary.  Let $t\in\N$. We first record the ball version of the cited Remez
inequality:
\[
 \|P\|_{L^2(B^d)}
 \lesssim \|P\|_{L^2(B^d\setminus B(0,(8t)^{-1}))},
 \qquad P\in\mathcal P_t.
\]
To see this, slice the ball parallel to the first coordinate. For
$y\in\R^{d-1}$ with $|y|<(8t)^{-1}$, the slice has half-length
$a_y=(1-|y|^2)^{1/2}>1/2$. The deleted central interval has
half-length at most $(8t)^{-1}\le a_y/(4t)$. Scale the slice to
$(-1,1)$ and apply \cite[Lemma~10]{KLM2008} with $q=2$.
For the other slices nothing is deleted. Integrating the squared
one-dimensional inequality in $y$ proves the displayed estimate with
a constant independent of $t$.

Choose the tail tolerance above sufficiently small and then $\delta$
sufficiently small that
$\|F_\delta\|_{L^2(B^d)}\ge\frac12|A|\delta^{d/2}\|\varphi\|_2$.
If $t\delta C_\varphi\le1/8$, the norm of $F_\delta$ outside
$B(0,(8t)^{-1})$ is at most $2\varepsilon\|F_\delta\|_{L^2(B^d)}$.
Renaming $2\varepsilon$ as $\varepsilon$, the Remez inequality gives
\[
 \|P\|_{L^2(B^d)}
 \lesssim
 \|F_\delta-P\|_{L^2(B^d)}
 +\varepsilon\|F_\delta\|_{L^2(B^d)}.
\]
Hence
\[
\|F_\delta\|_2
\le\|F_\delta-P\|_2+\|P\|_2
\le(1+C_R)\|F_\delta-P\|_2
+C_R\varepsilon\|F_\delta\|_2.
\]
Fix $\varepsilon<(2C_R)^{-1}$ and absorb the final term.  Thus
\[
E_2(F_\delta,\mathcal P_t;B^d)
\ge c\|F_\delta\|_2.
\]

Given $N\ge1$, choose $s=\lceil N^{1/(d-1)}\rceil$.  Since $N\le s^{d-1}$, monotonicity of best approximation with respect to the model class and the KLM comparison theorem imply
\[
E_2(F_\delta,\mathcal R_{N,2};B^d)
\ge E_2(F_\delta,\mathcal R_{s^{d-1},2};B^d)
\ge cE_2(F_\delta,\mathcal P_{\bar c_0s};B^d)
\ge c\|F_\delta\|_2,
\]
provided $\bar c_0s\delta C_\varphi\le1/8$.

Now use $x=x_0+Ry$ and $\delta=h/R$.  Translation and dilation convert arbitrary biased ridges in $x$ into arbitrary univariate ridge profiles in $y$ without changing their number.  If $N\le c_1h^{-(d-1)}$ with $c_1$ sufficiently small, the support condition above holds.  Finally
\[
\|\varphi_h^{(a)}\|_{L^2(B(x_0,R))}
\simeq h^{a+d/2}\|\varphi\|_2
\]
for all sufficiently small $h$.  This proves the assertion.
\end{proof}

To apply the scalar obstruction to odd-order Sobolev derivatives,
we need a difference operator controlled by one derivative. Spherical
averaging has this property and also preserves ridge directions.

\begin{lemma}[Spherical-mean difference bound]\label{lem:mean-diff}
Let $M_tF(x)=\int_{\Sph^{d-1}}F(x+t\theta)\,d\sigma(\theta)$, where $d\sigma$ is probability surface measure, and set
\[
T_tF:=\frac{I-M_t}{t}F.
\]
If $F\in H^1(B(x_0,2R))$ and $0<t<R/2$, then
\[
\|T_tF\|_{L^2(B(x_0,R))}
\lesssim \|\nabla F\|_{L^2(B(x_0,2R))},
\]
with $C$ independent of $t$.
\end{lemma}

\begin{proof}
The fundamental theorem of calculus gives
\[
\frac{F(x)-F(x+t\theta)}{t}
=-\int_0^1\theta\cdot\nabla F(x+st\theta)\,ds.
\]
Average in $\theta$.  Jensen's inequality first in $(s,\theta)$ and then Fubini give
\[
|T_tF(x)|^2
\le\int_0^1\int_{\Sph^{d-1}}|\nabla F(x+st\theta)|^2\,d\sigma(\theta)\,ds.
\]
Integrating over $B(x_0,R)$ and translating the integration variable yields a bound by the $L^2$ norm of $\nabla F$ over $B(x_0,2R)$, because $st<R/2$. The argument is first applied to smooth $F$ and
then passed to $H^1$ by density on the ball.
\end{proof}

The preceding estimate is useful only if the difference operator
does not enlarge the ridge count. Rotational invariance provides
exactly that preservation property.

\begin{lemma}[Spherical means preserve ridge directions]\label{lem:ridge-preserve}
If $Q\in L^2_{\rm loc}(\R)$, then for every unit $\omega$ and $b\in\R$ there is a univariate profile $\widetilde Q_t$ such that
\[
T_t[Q(\omega\cdot x+b)]=\widetilde Q_t(\omega\cdot x+b).
\]
Consequently $T_t$ maps a sum of $N$ ridges into a sum of at most $N$ ridges with the same directions.
\end{lemma}

\begin{proof}
Rotational invariance of $d\sigma$ implies that the distribution of $\omega\cdot\theta$ is independent of the chosen unit vector $\omega$.  Hence
\[
M_t[Q(\omega\cdot x+b)]
=\int_{\Sph^{d-1}}Q(\omega\cdot x+b+t\omega\cdot\theta)\,d\sigma(\theta)
=(Q*\mu_t)(\omega\cdot x+b)
\]
for a one-dimensional probability measure $\mu_t$.  Subtracting from $Q$ and dividing by $t$ gives another univariate profile.
\end{proof}

It remains to verify that applying this difference to a rescaled
radial potential produces a nonzero radial profile at the same scale.

\begin{lemma}[Scaling of a radial potential under $T_t$]\label{lem:potential-scale}
Let $V_h(x)=h^aV((x-x_0)/h)$ with fixed nonzero radial $V\in\mathcal S(\R^d)$.  There exists a sufficiently small fixed $\eta>0$ such that
\[
T_{\eta h}V_h(x)
=h^{a-1}\Psi_\eta\left(\frac{x-x_0}{h}\right),
\]
where $\Psi_\eta$ is a fixed nonzero radial Schwartz function.
\end{lemma}

\begin{proof}
Substitution gives
\[
M_{\eta h}V_h(x)
=h^a\int_{\Sph^{d-1}}V\left(\frac{x-x_0}{h}+\eta\theta\right)d\sigma(\theta),
\]
whence
\[
T_{\eta h}V_h(x)
=h^{a-1}\eta^{-1}\left[V(y)-\int_{\Sph^{d-1}}V(y+\eta\theta)d\sigma(\theta)\right]_{y=(x-x_0)/h}.
\]
The bracket defines a radial Schwartz function.  The spherical-mean Taylor expansion is
\[
M_\eta V=V+c_d\eta^2\Delta V+O(\eta^4)
\]
in the Schwartz topology.  A harmonic Schwartz function is zero, so
$\Delta V\not\equiv0$.  Thus the bracket is nonzero for all sufficiently
small fixed $\eta$.
\end{proof}

\subsection{Gevrey localization in direction and bias}\label{sec:gevrey}

The scalar radial obstruction does not see the bias parameter.  To obtain the
second branch, we need a transform whose response to one ridge atom is
localized simultaneously in direction and bias.  Compactly supported Gevrey
cutoffs provide the required subexponential Fourier decay.

Fix $0<\vartheta<1$ and put
\[
\sigma:=\frac1\vartheta>1.
\]
Choose once and for all $\chi\in G_c^\sigma(\Omega)$ such that $\chi\equiv1$ on a neighborhood of the compact set supporting all hard targets, and choose an even nonzero band-pass cutoff
\[
\eta\in G_c^\sigma(\R),
\qquad
\supp\eta\subset\{r:\tfrac12\le|r|\le2\}.
\]
The existence of compactly supported Gevrey cutoffs for every $\sigma>1$
and the derivative convention used here are classical; see
\cite[Chapter~1]{Rodino1993}.
Thus for suitable $C_0,C_1$,
\[
\|D^\alpha\chi\|_\infty\lesssim C_1^{|\alpha|}(\alpha!)^\sigma,
\qquad
\|\eta^{(\ell)}\|_\infty\lesssim C_1^\ell(\ell!)^\sigma.
\]

\begin{lemma}[One-dimensional Gevrey Fourier decay]\label{lem:gevrey-fourier}
Let $A,B>0$ and $E\ge0$. Suppose $g\in C_c^\infty(\R)$ is
supported in a fixed compact interval and satisfies
\[
\|g^{(\ell)}\|_\infty\le AE B^\ell(\ell!)^\sigma,
\qquad \ell\ge0.
\]
Then
\[
\left|\int_\R e^{iru}g(r)\,dr\right|
\lesssim Ee^{-c|u|^\vartheta}.
\]
\end{lemma}

This is the compactly supported Gevrey Fourier characterization
\cite[Chapter~1]{Rodino1993}, applied to $g/E$ when $E>0$.
The constants depend on $A,B,\sigma$ and the fixed support interval,
but not on $E$ or $u$. When $E=0$, the assertion is immediate.

The packet argument also creates powers of the frequency variable.
The next elementary estimate quantifies their absorption without
losing the Gevrey exponent.

\begin{lemma}[Polynomial absorption by a Gevrey tail]\label{lem:absorb}
For every $a>0$ there are $C,c>0$ such that
\[
x^\ell e^{-ax^\vartheta}
\lesssim C^{\ell+1}(\ell!)^\sigma e^{-cx^\vartheta},
\qquad x\ge0,\ \ell\in\N_0.
\]
\end{lemma}

\begin{proof}
Write
\[
x^\ell e^{-ax^\vartheta}
=\bigl(x^\ell e^{-ax^\vartheta/2}\bigr)e^{-ax^\vartheta/2}.
\]
Let $y=x^\vartheta$.  Since $1/\vartheta=\sigma$, $x^\ell=y^{\sigma\ell}$.  The function $y^{\sigma\ell}e^{-ay/2}$ is maximized at $y=2\sigma\ell/a$, and its maximum equals
\[
\left(\frac{2\sigma\ell}{ae}\right)^{\sigma\ell}.
\]
Stirling's formula bounds this by $C^{\ell+1}(\ell!)^\sigma$.  Multiplying by $e^{-ax^\vartheta/2}$ proves the claim.
\end{proof}

Fix $\omega_0\in\Sph^{d-1}$, write $x=s\omega_0+y$ with $y\in\omega_0^\perp$, and define the tangential partial Fourier transform
\[
A_{\omega_0}(s,\zeta)
:=\int_{\omega_0^\perp}e^{-iy\cdot\zeta}\chi(s\omega_0+y)\,dy,
\qquad \zeta\in\omega_0^\perp.
\]
For a unit $e\in\omega_0^\perp$ write $D_{e,\zeta}=e\cdot\nabla_\zeta$.

\begin{lemma}[Uniform Gevrey tangential decay]\label{lem:tangential}
There are constants $C,c>0$, independent of $\omega_0$, such that for
all integers $q,r\ge0$, unit vectors $e\in\omega_0^\perp$, and $(s,\zeta)$,
\[
 |\partial_s^qD_{e,\zeta}^rA_{\omega_0}(s,\zeta)|
 \lesssim C^{q+r+1}(q!)^\sigma e^{-c|\zeta|^\vartheta}.
\]
In particular no additional factor $(r!)^\sigma$ is required.
\end{lemma}

\begin{proof}
Differentiating under the integral gives
\[
 \partial_s^qD_{e,\zeta}^rA_{\omega_0}(s,\zeta)
 =(-i)^r\int_{\omega_0^\perp}e^{-iy\cdot\zeta}
 (y\cdot e)^rD_{\omega_0}^q\chi(s\omega_0+y)\,dy.
\]
All integration sets have uniformly bounded volume, and $|y|\lesssim1$ on
their support. For $\zeta\ne0$, integrate by parts $N$ times in the
direction $e_\zeta=\zeta/|\zeta|$. The Leibniz expansion contains terms
bounded, apart from $C^{N+q+r+1}$, by
\[
 \binom Nj\frac{r!}{(r-j)!}\bigl((N-j+q)!\bigr)^\sigma,
 \qquad 0\le j\le\min\{N,r\}.
\]
Here arbitrary directional derivatives of $\chi$ obey the same Gevrey
bounds as coordinate derivatives, with an additional exponential constant.
Use $(N-j+q)!\le2^{N+q}(N-j)!\,q!$ and, since $\sigma>1$,
\[
 \binom Nj\frac{r!}{(r-j)!}\bigl((N-j)!\bigr)^\sigma
 =\binom rj N!\bigl((N-j)!\bigr)^{\sigma-1}
 \le\binom rj(N!)^\sigma.
\]
Summation over $j$ costs at most $2^r$, rather than a factorial in $r$.
Consequently
\[
 |\partial_s^qD_{e,\zeta}^rA_{\omega_0}|
 \lesssim C^{N+q+r+1}(q!N!)^\sigma|\zeta|^{-N}.
\]
For $|\zeta|$ large choose $N=\lfloor c_0|\zeta|^{1/\sigma}\rfloor$.
The inequality $N!\le N^N$ and a sufficiently small $c_0$ make the
$N$-dependent factor at most $2^{-N}$. For bounded $|\zeta|$ the
undifferentiated integral supplies the bound. This proves the assertion
uniformly under rotations of $\omega_0^\perp$.
\end{proof}

This absence of a factorial in the Fourier-moment order is essential:
those derivatives will introduce powers of $|q|/h$ in the packet proof,
and absorbing those powers uses the available Gevrey factorial once.

\paragraph{Band-pass ridgelet analysis and localization.}\label{sec:packet}

We now incorporate the Gevrey cutoff into a band-pass ridgelet transform.
The next estimates have two roles: the transform is bounded by the
$H^m$ error, and the image of a single neuron is concentrated near its own
direction--bias parameter.

Define
\[
s_m^\sharp:=k-m+\frac{d+1}{2}.
\]
For $0<h\le1$ and $v\in H^m(\Omega)$ set
\[
A_hv(\omega,b)
:=c_d\int_\R e^{itb}\eta(ht)(it)^{k+1}|t|^{d-1-s_m^\sharp}\widehat{\chi v}(t\omega)\,dt.
\]
The product $\chi v$ is extended by zero to $\R^d$.

\begin{proposition}[$H^m$ boundedness of $A_h$]\label{prop:Ah-bounded}
There exists $C>0$, independent of $h$, such that
\[
\|A_hv\|_{L^2(\Sph^{d-1}\times\R)}
\lesssim \|v\|_{H^m(\Omega)}.
\]
\end{proposition}

\begin{proof}
Since
\[
k+1+d-1-s_m^\sharp=m+\frac{d-1}{2},
\]
one-dimensional Plancherel in $b$ yields
\[
\|A_hv\|_2^2
\simeq\int_{\Sph^{d-1}}\int_\R
|\eta(ht)|^2|t|^{2m+d-1}|\widehat{\chi v}(t\omega)|^2\,dt\,d\omega.
\]
Using polar coordinates $\xi=t\omega$ and absorbing the fixed two-to-one redundancy $(t,\omega)\sim(-t,-\omega)$,
\[
\|A_hv\|_2^2
\lesssim\int_{\R^d}|\xi|^{2m}|\widehat{\chi v}(\xi)|^2\,d\xi
\lesssim\|\chi v\|_{H^m(\R^d)}^2.
\]
Multiplication by the fixed compactly supported smooth function $\chi$ is bounded from $H^m(\Omega)$ to $H^m(\R^d)$, proving the result.
\end{proof}

For $\theta=(\omega,b)$ and $\theta'=(\omega',b')$ put
\[
d_\Theta(\theta,\theta'):=d_{\Sph}(\omega,\omega')+|b-b'|.
\]

\begin{theorem}[Localized Gevrey packet of one dictionary atom]\label{thm:packet}
Let $|\omega_0|=1$ and let $b_0$ range over a fixed bounded set.  Then
\[
|A_h[\sigma_k(\omega_0\cdot x-b_0)](\omega,b)|
\lesssim h^{\beta_m-d/2}\sum_{\eps=\pm1}
\exp\left[-c\left(
\frac{d_\Sph(\omega,\eps\omega_0)+|b-\eps b_0|}{h}
\right)^\vartheta\right].
\]
\end{theorem}

\begin{proof}
\textbf{Step 1: normal--tangential Fourier representation.}
Write $x=s\omega_0+y$, $y\perp\omega_0$, and decompose
$\xi=\lambda\omega_0+\zeta$, $\zeta\perp\omega_0$.  Since
$\sigma_k(\omega_0\cdot x-b_0)=(s-b_0)_+^k$,
\[
\widehat{\chi\,\sigma_k(\omega_0\cdot x-b_0)}(\xi)
=\int_\R e^{-i\lambda s}(s-b_0)_+^kA_{\omega_0}(s,\zeta)\,ds.
\]
With $u=s-b_0$,
\[
\widehat{\chi\,\sigma_k(\omega_0\cdot x-b_0)}(\xi)
=e^{-ib_0\lambda}J_k(\lambda,\zeta),
\]
where
\[
J_k(\lambda,\zeta)
:=\int_0^\infty u^ke^{-i\lambda u}A_{\omega_0}(b_0+u,\zeta)\,du.
\]
The integral is over a fixed finite interval because $\chi$ is compactly supported.

\textbf{Step 2: near-conormal derivative estimate.}
Put
\[
c:=\omega\cdot\omega_0,
\qquad
q:=P_{\omega_0^\perp}\omega.
\]
For $\xi=t\omega$, $\lambda=tc$ and $\zeta=tq$.  Assume first $|c|\ge1/2$.  On $1/2\le|r|\le2$ define
\[
a_h(r,\omega):=h^{-(k+1)}J_k\left(\frac{rc}{h},\frac{rq}{h}\right).
\]
We claim
\[
|\partial_r^\ell a_h(r,\omega)|
\lesssim C^{\ell+1}(\ell!)^\sigma
\exp\left[-c_1\left(\frac{|q|}{h}\right)^\vartheta\right].
\]
Differentiate $\ell$ times.  If $a$ derivatives hit the normal phase and $b=\ell-a$ hit the tangential argument, then
\[
\partial_r^\ell a_h
=h^{-(k+1+\ell)}\sum_{a+b=\ell}C_{a,b,\ell}c^a
\int_0^\infty u^{k+a}e^{-ircu/h}
(q\cdot\nabla_\zeta)^bA_{\omega_0}\left(b_0+u,\frac{rq}{h}\right)du.
\]
Set
\[
H_{a,b}(u):=u^{k+a}(q\cdot\nabla_\zeta)^bA_{\omega_0}\left(b_0+u,\frac{rq}{h}\right).
\]
Then $H_{a,b}^{(j)}(0)=0$ for $0\le j<k+a$.  Integrating by parts $k+a+1$ times in $u$, using $|r|\simeq1$ and $|c|\ge1/2$, gives
\[
\left|\int_0^\infty e^{-ircu/h}H_{a,b}(u)du\right|
\lesssim h^{k+a+1}
\left(|H_{a,b}^{(k+a)}(0)|+\|H_{a,b}^{(k+a+1)}\|_1\right).
\]
The prefactor becomes
\[
h^{-(k+1+\ell)}h^{k+a+1}=h^{-b}.
\]
If $q\ne0$, $(q\cdot\nabla_\zeta)^b=|q|^bD_{e_q,\zeta}^b$.  Set $K=k+a$. By Lemma~\ref{lem:tangential}, differentiation $j$
times in $u$ gives
\[
 |H_{a,b}^{(j)}(u)|
 \lesssim C^{K+b+j+1}|q|^b e^{-c_2(|q|/h)^\vartheta}
 \sum_{t=0}^{\min\{j,K\}}
 \binom jt\frac{K!}{(K-t)!}\bigl((j-t)!\bigr)^\sigma.
\]
The polynomial factor $u^{K-t}$ is bounded by $C^{K+1}$ on the fixed
integration interval. As in the preceding lemma,
\[
 \binom jt\frac{K!}{(K-t)!}\bigl((j-t)!\bigr)^\sigma
 \le \binom Kt(j!)^\sigma.
\]
Since $j\le k+a+1$ and $k$ is fixed, this yields
\[
 |H_{a,b}^{(j)}(u)|
 \lesssim C^{\ell+1}(a!)^\sigma |q|^b
 e^{-c_2(|q|/h)^\vartheta}.
\]
The same bound controls the boundary value and $L^1$ norm in the
integration-by-parts estimate. Multiplication by $h^{-b}$ introduces
$(|q|/h)^b$. Lemma~\ref{lem:absorb} absorbs this power at cost
$C^{b+1}(b!)^\sigma$ and a smaller exponential constant. Now
$a!b!\le\ell!$ because $a+b=\ell$. Summing the binomial coefficients
therefore proves
\[
 |\partial_r^\ell a_h(r,\omega)|
 \lesssim C^{\ell+1}(\ell!)^\sigma
 e^{-c_1(|q|/h)^\vartheta}.
\]
When $q=0$, the terms with $b>0$ vanish and the same conclusion follows.

\textbf{Step 3: insert the band-pass multiplier.}
Set
\[
M_0(r):=\eta(r)(ir)^{k+1}|r|^{d-1-s_m^\sharp}.
\]
Because $\eta$ is Gevrey and supported away from zero,
\[
|M_0^{(\ell)}(r)|\lesssim C^{\ell+1}(\ell!)^\sigma.
\]
Hence $G_h(r,\omega):=M_0(r)a_h(r,\omega)$ satisfies
\[
|\partial_r^\ell G_h(r,\omega)|
\lesssim C^{\ell+1}(\ell!)^\sigma e^{-c_3(|q|/h)^\vartheta}.
\]

\textbf{Step 4: bias localization.}
Substitute $t=r/h$ in $A_h$ and use the Fourier representation from Step 1.  The powers of $h$ are
\[
h^{k+1}h^{-(m+(d-1)/2)}h^{-1}
=h^{k-m-(d-1)/2}=h^{\beta_m-d/2}.
\]
Thus
\[
A_h[\sigma_k(\omega_0\cdot x-b_0)](\omega,b)
=c_dh^{\beta_m-d/2}
\int e^{ir(b-b_0c)/h}G_h(r,\omega)\,dr.
\]
Lemma~\ref{lem:gevrey-fourier} gives, for $|c|\ge1/2$,
\[
|A_h[\sigma_k(\omega_0\cdot x-b_0)](\omega,b)|
\lesssim h^{\beta_m-d/2}
 e^{-c_4(|q|/h)^\vartheta}
 e^{-c_4(|b-b_0c|/h)^\vartheta}.
\]

\textbf{Step 5: far-angular region.}
If $|c|<1/2$, then $|q|\ge\sqrt3/2$.  Here no integration by parts in the
normal variable is needed.  Differentiate the definition of $a_h$ directly.
Each of the $\ell$ derivatives produces either $cu/h$ from the normal
phase or $(q/h)\cdot\nabla_\zeta$ from the tangential argument.
Lemma~\ref{lem:tangential}, used with $q=0$ in its own notation, has
only an exponential constant in the number of Fourier derivatives.
Since $u$ ranges over a fixed bounded interval, it follows that
\[
 |\partial_r^\ell a_h(r,\omega)|
 \lesssim C^{\ell+1}h^{-(k+1+\ell)}
 e^{-c(|q|/h)^\vartheta}.
\]
For $|q|\ge\sqrt3/2$, apply Lemma~\ref{lem:absorb} to the power
$k+1+\ell$. The fixed shift $k+1$ in the factorial can be absorbed
into $C^{\ell+1}$, giving
\[
 |\partial_r^\ell a_h(r,\omega)|
 \lesssim C^{\ell+1}(\ell!)^\sigma e^{-c_5h^{-\vartheta}}.
\]
Leibniz' rule gives the same estimate for
$G_h=M_0a_h$.  Lemma~\ref{lem:gevrey-fourier}, applied to the $r$ integral
in Step~4, now gives
\[
|A_h[\sigma_k(\omega_0\cdot x-b_0)](\omega,b)|
\lesssim h^{\beta_m-d/2}e^{-c_5h^{-\vartheta}}
 e^{-c_5|b-b_0c|^\vartheta/h^\vartheta},
\]
which is stronger than required because $d_\Sph(\omega,\{\pm\omega_0\})$ is then bounded below.

\textbf{Step 6: replace the curved center by the two antipodal points.}
Suppose $\omega$ is close to $\eps\omega_0$, $\eps\in\{\pm1\}$, and put $\theta=d_\Sph(\omega,\eps\omega_0)$.  Then
\[
|q|\simeq\theta,
\qquad
|c-\eps|\lesssim \theta^2.
\]
Since $b_0$ stays bounded,
\[
|b-\eps b_0|
\le|b-b_0c|+C\theta^2
\le|b-b_0c|+C\theta.
\]
Hence
\[
\frac{\theta+|b-\eps b_0|}{h}
\lesssim \left(\frac{|q|}{h}+\frac{|b-b_0c|}{h}\right).
\]
For $0<\vartheta<1$, $x^\vartheta+y^\vartheta\ge(x+y)^\vartheta$.  Combining the near and far estimates and summing over $\eps=\pm1$ proves the theorem.
\end{proof}

The preceding bound describes one neuron's localized response.
For the targets, radial symmetry gives an exact response whose
amplitude and width can be computed without an estimate.

\begin{lemma}[Exact packet of a radial critical bump]\label{lem:exact-bump}
Choose a nonzero radial function
$\psi\in C_c^\infty(B(0,r_0))$.  Require that the annular multiplier below
is not identically zero, and let
\[
u_{h,x_0}(x)=h^k\psi\left(\frac{x-x_0}{h}\right),
\]
with $\chi\equiv1$ on its support.  Then
\[
A_hu_{h,x_0}(\omega,b)
=h^{\tau_m-1/2}
\Psi\left(\frac{b-\omega\cdot x_0}{h}\right),
\]
where
\[
\Psi(s)
:=c_d\int_\R e^{irs}\eta(r)(ir)^{k+1}|r|^{d-1-s_m^\sharp}\widehat\psi_0(|r|)\,dr
\]
is a fixed nonzero Schwartz function.  Consequently
\[
\|A_hu_{h,x_0}\|_{L^2(\Sph^{d-1}\times\R)}\simeq h^{\tau_m}.
\]
\end{lemma}

\begin{proof}
The Fourier transform of the translated and scaled bump is
\[
\widehat u_{h,x_0}(t\omega)
=h^{k+d}e^{-it\omega\cdot x_0}\widehat\psi(ht\omega).
\]
Radiality gives $\widehat\psi(ht\omega)=\widehat\psi_0(|ht|)$.  Substitute into the definition of $A_h$ and set $r=ht$.  The total power of $h$ is
\[
h^{k+d}h^{-(m+(d-1)/2)}h^{-1}
=h^{k-m+(d-1)/2}=h^{\tau_m-1/2},
\]
which yields the displayed formula.  The $r$-integrand is compactly supported and, by choice of $\psi$, not identically zero, so Fourier injectivity gives $\Psi\ne0$ and smooth compact support in frequency gives $\Psi\in\mathcal S$.

Finally, for fixed $\omega$ change variables $s=(b-\omega\cdot x_0)/h$:
\[
\int_\R|A_hu_{h,x_0}(\omega,b)|^2db
=h^{2\tau_m}\|\Psi\|_2^2.
\]
Integration over the sphere gives the asserted equivalence.
\end{proof}

\paragraph{Direction--bias hard families.}\label{sec:refined}

The packet transform turns the remaining lower-bound problem into a cell
count.  We place critical bumps along an interior cylinder, associate one
packet cell with each bump and angular direction, and show that a width-$n$
network can influence only a controlled number of those cells.

Let
\[
\mathcal C
=\{x_*+te_1+y:0\le t\le T,\ y\perp e_1,\ |y|\le r_*\}\Subset\Omega.
\]
Fix $s_*\in\R$ with $|\Psi(s_*)|>0$.  By continuity choose $a_*,c_\Psi>0$ such that
\[
|\Psi(s)|\ge2c_\Psi
\quad\text{whenever }|s-s_*|\le2a_*.
\]
Choose a spherical patch
\[
U\Subset\{\omega\in\Sph^{d-1}:\omega_1\ge c_0\}
\]
with $c_0>0$.

\section{Proof of the lower bounds and completion of the main theorem}
\label{sec:lower-proof}

We now turn the analytic localization estimates of
Section~\ref{sec:lower-lemmas} into quantitative lower bounds.  The proof is
organized around the geometry of parameter-space cells.  We first establish
the angular and direction--bias estimates, and finally combine the two lower branches
with the upper theorem. Each hard target may depend on the width, as allowed
by the minimax definition.

\subsection{The angular lower bound}

The scalar radial obstruction already applies to even-order Laplacians.
For odd Sobolev orders we first lift it to radial gradients by a
spherical-mean difference. This step preserves the number of ridge
directions.

\begin{theorem}[Vector radial-gradient obstruction]\label{thm:vector-obstruction}
Fix $R>0$ and $x_0$ with $B(x_0,4R)\Subset\Omega$.
Let $V_h(x)=h^aV((x-x_0)/h)$ with fixed nonzero radial $V\in\mathcal S(\R^d)$.  Let
\[
\mathcal V_N
:=\left\{\sum_{j=1}^{N}\omega_jq_j(\omega_j\cdot x+b_j):
|\omega_j|=1,\ q_j\in L^2_{\rm loc}(\R)\right\}.
\]
There exist $c_0,c_1,h_0>0$ such that, whenever $0<h<h_0$ and $N\le c_1h^{-(d-1)}$,
\[
\inf_{G\in\mathcal V_N}\|\nabla V_h-G\|_{L^2(B(x_0,2R))}
\ge c_0h^{a-1+d/2}.
\]
\end{theorem}

\begin{proof}
Given $G=\sum_{j=1}^{N}\omega_jq_j(\omega_j\cdot x+b_j)$, choose locally absolutely continuous primitives $Q_j$ with $Q_j'=q_j$ and set
\[
R_N(x)=\sum_{j=1}^{N}Q_j(\omega_j\cdot x+b_j).
\]
Then $\nabla R_N=G$ almost everywhere.  Apply Lemma~\ref{lem:mean-diff} to $V_h-R_N$ with $t=\eta h$:
\[
\|\nabla V_h-G\|_{L^2(B(x_0,2R))}
\ge c\|T_{\eta h}(V_h-R_N)\|_{L^2(B(x_0,R))}.
\]
By Lemma~\ref{lem:ridge-preserve}, $T_{\eta h}R_N$ is a sum of at most $N$ scalar ridge profiles.  By Lemma~\ref{lem:potential-scale}, $T_{\eta h}V_h$ is a fixed nonzero radial bump of scale $h$ and amplitude $h^{a-1}$.  Lemma~\ref{lem:scaled-klm} applied with exponent $a-1$ gives the stated lower bound.
\end{proof}

The scalar obstruction for even derivatives and the gradient
obstruction for odd derivatives now yield a uniform lower bound
in every integer Sobolev norm under consideration.

\begin{theorem}[Angular lower bound for finite dictionary networks]\label{thm:angular-lower}
Let $d\ge2$, $0<p<1$, $0<q\le1$, and $0\le m\le k$.  For every $M\in[0,\infty]$ there is $c>0$, independent of $n$ and $M$, such that, for every integer $n\ge1$,
\[
E_{n,m}(B_{B_{p,q}^{k+d/p}(\Omega)};M)
\ge cn^{-A_m}.
\]
\end{theorem}

\begin{proof}
Fix $R>0$ and $x_0$ with $B(x_0,4R)\Subset\Omega$.
Choose a nonzero radial $\psi\in C_c^\infty(B(0,1))$ and define $\psi_h$ as in Section~\ref{sec:scaling}.  Lemma~\ref{lem:single-besov} gives a uniform critical Besov norm.

If $m=2r$ is even, replace $\psi$ if necessary so that $\Delta^r\psi\not\equiv0$.  Then
\[
\Delta^r\psi_h(x)=h^{k-m}(\Delta^r\psi)\left(\frac{x-x_0}{h}\right).
\]
Distributional differentiation of a dictionary ridge gives
\[
\Delta^r(\omega\cdot x-b)_+^k=c_{k,m}(\omega\cdot x-b)_+^{k-m},
\]
which is an ordinary locally square-integrable univariate profile because $m\le k$.  Thus for every $g\in\Sigma_{n,M}$, $\Delta^rg$ is a sum of at most $n$ arbitrary ridge profiles.  The differential operator is bounded from $H^m$ to $L^2$ on an interior ball, hence
\[
\|\psi_h-g\|_{H^m(\Omega)}
\ge c\|\Delta^r(\psi_h-g)\|_{L^2(B(x_0,R))}.
\]
Lemma~\ref{lem:scaled-klm}, with amplitude exponent $k-m$, yields
\[
\|\psi_h-g\|_{H^m}\ge ch^{k-m+d/2}=ch^{\tau_m}
\]
whenever $n\le c_1h^{-(d-1)}$.

If $m=2r+1$ is odd, set $V_h=\Delta^r\psi_h$.  Then
\[
V_h(x)=h^{k-m+1}V\left(\frac{x-x_0}{h}\right)
\]
with fixed nonzero radial $V=\Delta^r\psi$.  Moreover
\[
\nabla\Delta^r(\omega\cdot x-b)_+^k
=c_{k,m}\omega(\omega\cdot x-b)_+^{k-m},
\]
so $\nabla\Delta^rg\in\mathcal V_n$.  Theorem~\ref{thm:vector-obstruction} with $a=k-m+1$ yields the same lower bound $ch^{\tau_m}$.

Choose a dyadic $h_n\simeq n^{-1/(d-1)}$ small enough for the corresponding counting condition, and normalize $f_n=C_B^{-1}\psi_{h_n}$ using Lemma~\ref{lem:single-besov}.  Then
\[
\sigma_{n,m}(f_n;M)
\ge ch_n^{\tau_m}
\simeq n^{-\tau_m/(d-1)}=n^{-A_m}.
\]
No coefficient bound was used, so the constant is uniform in $M$, including $M=\infty$.
\end{proof}

The angular bound is independent of the coefficient budget.
The remaining branch uses the bounded mass to exclude the small
Gevrey tails of all active neurons simultaneously.

\subsection{Direction--bias cell counting}

\paragraph{Logarithmically separated bump rows.}

For the refined construction define
\[
R_h:=A(\log(1/h))^{1/\vartheta},
\qquad
L_h:=L_*R_h,
\]
where $A,L_*>1$ are fixed sufficiently large constants.  Place the centers
\[
x_j=x_*+jL_hhe_1,
\qquad
1\le j\le K_h,
\qquad
K_h:=\left\lfloor\frac{c_*}{hR_h}\right\rfloor\simeq(hR_h)^{-1},
\]
where $c_*$ is small enough that all bump supports stay inside $\mathcal C$.  Define
\[
u_{h,j}(x)=h^k\psi\left(\frac{x-x_j}{h}\right),
\qquad
F_h:=K_h^{-1/p}\sum_{j=1}^{K_h}\eps_j u_{h,j}.
\]
By Lemma~\ref{lem:row-besov}, after one fixed rescaling of $\psi$ we may assume
\[
\|F_h\|_{B_{p,q}^{k+d/p}(\Omega)}\le1
\]
for all small dyadic $h$ and all signs.

\begin{lemma}[Suppression of inter-bump packet interference]\label{lem:interference}
If $L_*$ is sufficiently large, then for all sufficiently small $h$,
\[
\sup_{\substack{|s-s_*|\le a_*\\ \omega\in U}}
\sum_{r\in\mathbb Z\setminus\{0\}}
|\Psi(s+rL_h\omega_1)|\le c_\Psi.
\]
\end{lemma}

\begin{proof}
Since $\Psi\in\mathcal S(\R)$, for every $N>1$,
\[
|\Psi(t)|\lesssim (1+|t|)^{-N}.
\]
For $|s-s_*|\le a_*$ and $\omega_1\ge c_0$, choose $L_*$ so large that for every $r\ne0$,
\[
|s+rL_h\omega_1|
\ge \frac12|r|L_hc_0
\]
for all small $h$.  Then
\[
\sum_{r\ne0}|\Psi(s+rL_h\omega_1)|
\lesssim L_h^{-N}\sum_{r\ne0}|r|^{-N}.
\]
Since $L_h=L_*R_h\ge L_*$, the right side is made at most $c_\Psi$ by choosing $L_*$ sufficiently large.
\end{proof}

Partition $U$ into shape-regular cells $U_{\ell,h}$ of diameter $\simeq h$ and surface measure $\simeq h^{d-1}$.  Their number satisfies
\[
P_h\simeq h^{-(d-1)}.
\]
Define the direction--bias cells
\[
Q_{\ell,j}
:=\left\{(\omega,b):\omega\in U_{\ell,h},\ 
\left|\frac{b-\omega\cdot x_j}{h}-s_*\right|<a_*\right\}.
\]

\begin{lemma}[Cell geometry for refined spacing]\label{lem:refined-cells}
For $L_*$ sufficiently large the cells $Q_{\ell,j}$ are pairwise disjoint up to null boundaries and satisfy
\[
|Q_{\ell,j}|\simeq h^d,
\qquad
M_h:=\#\{Q_{\ell,j}\}=P_hK_h\simeq h^{-d}R_h^{-1}.
\]
\end{lemma}

\begin{proof}
For fixed $\omega$, the bias interval defining $Q_{\ell,j}$ has length $2a_*h$, so
\[
|Q_{\ell,j}|\simeq h\,|U_{\ell,h}|\simeq h^d.
\]
If $\ell\ne\ell'$, the angular cells are disjoint up to boundaries.  For fixed $\ell$ and $j\ne j'$, the bias centers differ by
\[
|\omega\cdot(x_j-x_{j'})|
=|j-j'|L_hh\,\omega_1
\ge c_0L_hh.
\]
For $L_*R_hc_0>4a_*$ the corresponding bias intervals are disjoint.  Finally
\[
M_h\simeq h^{-(d-1)}(hR_h)^{-1}=h^{-d}R_h^{-1}.
\]
\end{proof}

Disjoint parameter cells are not enough for a lower bound. Each
cell must also carry a target signal bounded away from zero, uniformly
in the signs of the bumps.

\begin{lemma}[Uniform target energy on every refined cell]\label{lem:target-cell}
For all sufficiently small $h$ and every target cell,
\[
|A_hF_h(\omega,b)|
\ge c h^{dG_{p,m}}R_h^{1/p},
\qquad (\omega,b)\in Q_{\ell,j}.
\]
\end{lemma}

\begin{proof}
Lemma~\ref{lem:exact-bump} gives
\[
A_hF_h(\omega,b)
=K_h^{-1/p}h^{\tau_m-1/2}
\sum_{j'=1}^{K_h}\eps_{j'}
\Psi\left(\frac{b-\omega\cdot x_{j'}}{h}\right).
\]
On $Q_{\ell,j}$ the $j$th argument lies within $a_*$ of $s_*$ and hence its absolute contribution is at least $2c_\Psi$.  By Lemma~\ref{lem:interference}, the sum of the absolute values of all other terms is at most $c_\Psi$.  Thus
\[
|A_hF_h|
\ge cK_h^{-1/p}h^{\tau_m-1/2}.
\]
Since $K_h\simeq(hR_h)^{-1}$,
\[
K_h^{-1/p}h^{\tau_m-1/2}
\simeq h^{\tau_m+1/p-1/2}R_h^{1/p}
=h^{dG_{p,m}}R_h^{1/p}.
\]
\end{proof}

For each angular cell choose a representative $\omega_{\ell,h}$ and define
\[
z_{\ell,j}:=(\omega_{\ell,h},\omega_{\ell,h}\cdot x_j+hs_*).
\]
There is a fixed $C_0$ such that $Q_{\ell,j}\subset B_\Theta(z_{\ell,j},C_0h)$.

\begin{lemma}[One neuron is near only $O(R_h^{d-1})$ refined cells]\label{lem:refined-count}
For every neuron parameter $\theta=(\omega_0,b_0)$,
\[
\#\{(\ell,j):d_\Theta(z_{\ell,j},\{\theta,-\theta\})\le R_hh\}
\lesssim R_h^{d-1}.
\]
\end{lemma}

\begin{proof}
Consider one of the two antipodal centers.  The angular condition confines $\omega_{\ell,h}$ to a spherical ball of radius $O(R_hh)$.  Since the angular net spacing is $h$, such a ball contains at most $CR_h^{d-1}$ angular cells.  Fix one of them.  As $j$ varies, the bias centers have spacing
\[
L_hh\,\omega_{\ell,h,1}
\ge c_0L_*R_hh.
\]
An interval of length $O(R_hh)$ therefore contains only $O(1)$ indices $j$.  Multiplying the two counts gives $CR_h^{d-1}$.  The antipodal center changes only the constant.
\end{proof}

The signal estimate and the cell count can now be combined. The
fixed coefficient budget makes the sum of all far-neuron tails small,
which is the step where this lower bound differs from the angular one.

\begin{theorem}[Refined fixed-budget $p$-dependent lower bound]\label{thm:refined-packet-lower}
Let $d\ge2$, $0<p<1$, $0<q\le1$, $0\le m\le k$, and assume that $\Omega$ contains a nontrivial closed cylinder.  Fix $0\le M<\infty$ and $0<\vartheta<1$.  Then for all $n\ge2$,
\[
E_{n,m}(B_{B_{p,q}^{k+d/p}(\Omega)},M)
\ge c\,n^{-G_{p,m}}(1+\log(2+n))^{-\tau_m/\vartheta}.
\]
Consequently, for every $\eps>0$,
\[
 E_{n,m}(B_{B_{p,q}^{k+d/p}(\Omega)};M)
 \ge c_{\eps,M}n^{-G_{p,m}}(1+\log(2+n))^{-\tau_m-\eps}.
\]
\end{theorem}

\begin{proof}
Let
\[
 g=\sum_{r=1}^{N}a_r\sigma_k(\omega_r\cdot x-b_r)\in\Sigma_{n,M},
\qquad N\le n.
\]
Call a target cell $R_h$-near if its center lies within $R_hh$ of one of the
two packet centers $\{(\omega_r,b_r),(-\omega_r,-b_r)\}$ for some neuron.
Lemma~\ref{lem:refined-count} gives
\[
\#\mathcal N_h(g)\lesssim nR_h^{d-1}.
\]
The total number of target cells is $M_h\simeq h^{-d}R_h^{-1}$.  Therefore, if
\[
nR_h^d\le c_*h^{-d},
\]
with $c_*$ sufficiently small, at least a fixed proportion of all cells are far:
\[
\#\mathcal F_h(g)\ge ch^{-d}R_h^{-1}.
\]

Let $Q_{\ell,j}$ be far and $(\omega,b)\in Q_{\ell,j}$.  Since the cell diameter is $O(h)$ and $R_h\to\infty$, for small $h$
\[
 d_\Theta((\omega,b),
 \{(\omega_r,b_r),(-\omega_r,-b_r)\})\ge cR_hh
\]
for every neuron.  The packet localization theorem therefore yields
\[
|A_hg(\omega,b)|
\le\sum_{r=1}^{N}|a_r|\,
 |A_h[\sigma_k(\omega_r\cdot x-b_r)](\omega,b)|
\lesssim Mh^{\beta_m-d/2}e^{-cR_h^\vartheta}.
\]
Because $R_h=A(\log(1/h))^{1/\vartheta}$,
\[
e^{-cR_h^\vartheta}=h^{cA^\vartheta}.
\]
Moreover
\[
dG_{p,m}-\left(\beta_m-\frac d2\right)
=d+\frac1p-1>0.
\]
Choose $A$ so large that $cA^\vartheta>d+1/p$.  Then
\[
|A_hg|=o\bigl(h^{dG_{p,m}}R_h^{1/p}\bigr).
\]
Combining this with Lemma~\ref{lem:target-cell}, for all small $h$,
\[
|A_h(F_h-g)(\omega,b)|
\ge ch^{dG_{p,m}}R_h^{1/p}
\]
on every far cell.

The far cells are disjoint, each has measure $\simeq h^d$, and there are $\gtrsim h^{-d}R_h^{-1}$ of them.  Their union therefore has measure $\gtrsim R_h^{-1}$.  Hence
\[
\|A_h(F_h-g)\|_2
\ge ch^{dG_{p,m}}R_h^{1/p}R_h^{-1/2}
=ch^{dG_{p,m}}R_h^{1/p-1/2}.
\]
Proposition~\ref{prop:Ah-bounded} gives
\[
\|F_h-g\|_{H^m(\Omega)}
\ge ch^{dG_{p,m}}R_h^{1/p-1/2}.
\]

The counting condition is ensured by
\[
h^{-d}\ge Cn(\log(1/h))^{d/\vartheta}.
\]
Choose dyadic $h_n$ comparable to
\[
h_n\simeq n^{-1/d}(1+\log(2+n))^{-1/\vartheta}.
\]
Then $R_{h_n}\simeq(1+\log(2+n))^{1/\vartheta}$ and Lemma~\ref{lem:row-besov} makes $F_{h_n}$ an admissible unit-ball target.  Thus
\[
\sigma_{n,m}(F_{h_n};M)
\gtrsim
h_n^{dG_{p,m}}R_{h_n}^{1/p-1/2}.
\]
Since
\[
dG_{p,m}-\left(\frac1p-\frac12\right)=\tau_m,
\]
the logarithmic power simplifies to
\[
-\frac{dG_{p,m}}\vartheta+\frac{1/p-1/2}{\vartheta}
=-\frac{\tau_m}{\vartheta},
\]
proving the first estimate.

For the second assertion, given $\eps>0$ choose
\[
\vartheta=\frac{\tau_m}{\tau_m+\eps}\in(0,1).
\]
Then $\tau_m/\vartheta=\tau_m+\eps$.
\end{proof}

\paragraph{Optimal algebraic exponent in one-scale bump--cell constructions.}
\label{sec:one-scale}

\begin{remark}[Interpretation of the two exponents]\label{prop:alpha-opt}
Suppose a one-scale hard family contains
\[
K_h\simeq h^{-\alpha},
\qquad 0\le\alpha\le1,
\]
separated critical $h$-bumps and uses all $h^{-(d-1)}$ angular cells generated by each bump.  Then the algebraic lower-bound exponent generated by this mechanism is
\[
\gamma(\alpha)
=\frac{\tau_m+\alpha(1/p-1/2)}{d-1+\alpha},
\]
and
\[
\min_{0\le\alpha\le1}\gamma(\alpha)
=\min\{A_m,G_{p,m}\}.
\]

The number of direction--bias cells is
\[
N_h\simeq K_hh^{-(d-1)}=h^{-(d-1+\alpha)}.
\]
After $\ell^p$ normalization, disjointness of the physical bumps gives the $H^m$ scale
\[
K_h^{1/2-1/p}h^{\tau_m}
=h^{\tau_m+\alpha(1/p-1/2)}.
\]
Ignoring logarithmic packet-exclusion factors, the critical width is $n\simeq N_h$, hence
\[
h\simeq n^{-1/(d-1+\alpha)}.
\]
Substitution produces $n^{-\gamma(\alpha)}$ with the displayed $\gamma$.

Let $a=1/p-1/2$.  A direct differentiation gives
\[
\gamma'(\alpha)
=\frac{a(d-1)-\tau_m}{(d-1+\alpha)^2},
\]
whose sign is independent of $\alpha$.  Thus the minimum on $[0,1]$ occurs at an endpoint:
\[
\gamma(0)=\frac{\tau_m}{d-1}=A_m,
\qquad
\gamma(1)=\frac{\tau_m+1/p-1/2}{d}=G_{p,m}.
\]
The restriction $\alpha\le1$ refers to this row construction with
disjoint bias cells for every angular direction. A bounded bias interval
contains only $O(h^{-1})$ disjoint $h$-cells. This calculation is not an
optimality statement for arbitrary multiscale hard families.
\end{remark}

\subsection{Completion of the proof of the main theorem}\label{sec:main-proof}

\begin{proof}[Proof of Theorem~\ref{thm:main-sharp}]
\textbf{Step 1: Stable upper estimate.}
Let $C_u$ be the unit-ball coefficient budget in
Theorem~\ref{thm:pure-upper}, and set $M_0=C_u$.  If
$\|f\|_{B_{p,q}^{k+d/p}(\Omega)}\le1$, then
$C_u\|f\|_{B_{p,q}^{k+d/p}(\Omega)}\le M_0$.  Hence monotonicity of the
classes $\Sigma_{n,M}$ and Theorem~\ref{thm:pure-upper} give
\begin{equation}
 E_{n,m}(B_{B_{p,q}^{k+d/p}(\Omega)};M_0)
 \lesssim  n^{-\Gamma_{p,m}}L_n^{\kappa_{p,q,m}},
 \label{eq:main-upper}
\end{equation}
where $\kappa_{p,q,m}$ is the piecewise exponent in that theorem.  In
particular, the same value of $M_0$ works for every width $n$.

\textbf{Step 2: The two lower estimates.}
The angular obstruction in Theorem~\ref{thm:angular-lower} is independent
of the coefficient budget and yields
\begin{equation}
 E_{n,m}(B_{B_{p,q}^{k+d/p}(\Omega)};M_0)
 \ge c n^{-A_m}.
 \label{eq:main-angular}
\end{equation}
For every fixed $0<\vartheta<1$, the direction--bias construction in
Theorem~\ref{thm:refined-packet-lower}, applied with the same $M_0$, gives
\begin{equation}
 E_{n,m}(B_{B_{p,q}^{k+d/p}(\Omega)};M_0)
 \ge c_\vartheta n^{-G_{p,m}}L_n^{-\tau_m/\vartheta}.
 \label{eq:main-packet}
\end{equation}

\textbf{Step 3: Identification of the phase transition.}
Using $A_m=G_{p_\ast,m}$ and the definitions of the two exponents,
we have
\[
 G_{p,m}-A_m
 =\frac1d\left(\frac1p-\frac1{p_\ast}\right).
\]
If $p<p_\ast$, then $G_{p,m}>A_m$,
$\Gamma_{p,m}=A_m$, and $\kappa_{p,q,m}=0$.
Equations~\eqref{eq:main-upper} and~\eqref{eq:main-angular} therefore give
$E_{n,m}\simeq n^{-A_m}$, proving part~(i).  If
$p=p_\ast$, the same angular lower bound and the endpoint case of
\eqref{eq:main-upper} give exactly the two inequalities in part~(ii).
Finally, if $p>p_\ast$, then $G_{p,m}<A_m$ and
$\Gamma_{p,m}=G_{p,m}$.  Combining~\eqref{eq:main-packet} with the last
case of~\eqref{eq:main-upper} gives part~(iii), including both logarithmic
powers.

This completes the proof.
\end{proof}

\section*{Conclusion}
\addcontentsline{toc}{section}{Conclusion}

We have determined the optimal algebraic $H^m$ approximation exponents
for critical Besov balls under a fixed coefficient budget. The direction
sphere and joint direction--bias space yield distinct obstructions,
matched by wavelet sparsity and stable local discretization.

The bounds match without logarithmic loss for $p<p_\ast$ and for
$p=p_\ast$ with $q\le(A_m+1/2)^{-1}$. The other logarithmic gaps remain
open. Critical smoothness isolates variation control without extra
scale decay. Above this line, the upper bounds persist by embedding
but need not be optimal. Below it, fixed-budget representability is a
different question from approximation with an increasing budget.

\section*{Acknowledgments}
\addcontentsline{toc}{section}{Acknowledgments}

The author is grateful to Professor Yuwen Li for suggesting the problem studied in
this paper and for his valuable advice.

\end{document}